\documentclass[11pt]{amsart}
\usepackage{xcolor}
\usepackage{amsmath}
\usepackage{amssymb}
\usepackage{tikz}
\usepackage{relsize}
\usepackage{blkarray}
\usepackage{mathrsfs}
\usepackage{comment}
\usepackage{hyperref}

\hypersetup{colorlinks=true,linkcolor=blue,citecolor=red}

\newtheorem{Definition}{Definition}[section]
\newtheorem{Theorem}[Definition]{Theorem}
\newtheorem{Lemma}[Definition]{Lemma}
\newtheorem{Proposition}[Definition]{Proposition}
\newtheorem{Corollary}{Corollary}[section]

\theoremstyle{definition}
\newtheorem{Obs}{Observation}[section]
\newtheorem{example}{Examples}[section]
\theoremstyle{definition}
\newtheorem{Remark}{Remark}[section]

\newcommand{\C}{\mathbb C}
\newcommand{\Z}{\mathbb Z}

\newcommand{\N}{\mathbb N}

\allowdisplaybreaks[1]
\begin{document}
\baselineskip16pt

\author[Sumit Kumar Rano and Rudra P. Sarkar]{Sumit Kumar Rano and Rudra P. Sarkar}

\address{Sumit Kumar Rano \endgraf Stat-Math Unit,	\endgraf Indian Statistical Institute,	\endgraf 203 B. T. Road, Kolkata 700108, India.} \email{sumitrano1992@gmail.com}

\address{Rudra P. Sarkar  \endgraf Stat-Math Unit,	\endgraf Indian Statistical Institute, 	\endgraf 203 B. T. Road, Kolkata 700108, India.} \email{rudra.sarkar@gmail.com, rudra@isical.ac.in}

\title[A Theorem of Strichartz for Multipliers on homogeneous trees]
{A Theorem of Strichartz for Multipliers on Homogeneous Trees}

\begin{abstract}
A theorem of Strichartz states that if a  uniformly bounded bi-infinite sequence of functions on the Euclidean spaces, satisfies the condition that the Laplacian acting on a function in this sequence yields the next one, then each function in this sequence is an eigenfunction of the Laplacian. We consider a generalization of this result for homogeneous trees, where we replace bounded functions by tempered distributions and the Laplacian by multiplier operators acting on the tempered distributions. After establishing the result in this general context, we narrow our focus to specific cases, which includes important examples of multiplier operators such as heat and Schr\"odinger operator, ball and sphere averages of function.
\end{abstract}


\maketitle


\section{Introduction}
Inspired by a result of Roe \cite{R}, Strichartz \cite{S} proved the following theorem.

\begin{Theorem}\label{th2} Suppose that a
bi-infinite sequence $\{f_{k}\}_{k\in\mathbb{Z}}$ of functions on $\mathbb{R}^{n}$ satisfies $\Delta_{\mathbb{R}^{n}} f_{k}=f_{k+1}$ and
$\|f_{k}\|_{L^{\infty}(\mathbb{R}^{n})}\leq M$, for all $k\in\mathbb{Z}$ and for a constant $M>0$. Then $\Delta_{\mathbb{R}^{n}} f_{0}=-f_{0}$.
\end{Theorem}
 This essentially characterizes all bounded eigenfunctions of the standard Euclidean Laplacian $\Delta_{\mathbb{R}^{n}}$ for which the corresponding eigenvalues are negative.  Strichartz \cite{S} also observed that an analogue of Theorem \ref{th2}  fails for hyperbolic 3-space. Taking the cue from there  Kumar et. al. \cite{KRS2} showed how to generalize this result for all noncompact Riemannian symmetric spaces of rank one. For homogeneous trees an analogue of Theorem \ref{th2} is obtained in \cite{R3}.

In this article we revisit  results of this genre on a homogeneous tree $\mathcal{X}$. Our point of departure is to  view the Laplace operator $\mathcal{L}$ on $\mathcal{X}$ as a translation invariant continuous linear operator acting on the Schwartz spaces, and by duality on the tempered distributions on $\mathcal X$. We shall refer to such operators as {\em multipliers}.  Our aim is to establish an analogue of Theorem \ref{th2} for trees, where we replace the bounded functions $f_k$ and the Laplace operator with  tempered distributions and  multipliers respectively, yielding a characterization of eigendistributions of these multipliers.
Theorems \ref{maximumcase} to \ref{maximumcasezero} are our endeavors in this pursuit.

Having established our results in this general set up, we restrict our attention to  particular cases. Our ideal candidates for $L^p$-tempered distributions are weak-$L^p$-functions on $\mathcal X$ and that for multipliers are $\Psi(\mathcal L)$,  where $\Psi$ is a holomorphic function defined on an open set containing the $L^p$-spectrum of the Laplace operator $\mathcal L$ on $\mathcal{X}$. The purpose of focusing on these particular cases  is that here  we can strengthen our conclusions from eigenfunctions of multipliers to eigenfunctions of the Laplacian, which can be concretely realized as Poisson transform of $L^p$-functions on the boundary of $\mathcal X$.

As important examples we take  $\Psi$ to be  the complex-time heat operator, (which includes both the usual heat operator and the Schr\"odinger operator) and  the sphere and ball averages of a function.

The choice of weak $L^p$-functions with $p>2$ is not arbitrary. It can be justified observing that the elementary spherical functions corresponding to the points on the boundary of $L^p$-spectrum are  weak-$L^p$ but not $L^p$-functions on $\mathcal{X}$.  Alternatively, one can use functions satisfying Hardy-type norm estimates (see \cite{R4}) to achieve the same results (see  Subsection \ref{remark on hardy type functions}).

It may not be out of place to note that, despite being regarded  as the discrete analogue of the hyperbolic spaces, the results on rank one symmetric spaces do not imply corresponding results for trees. There are significant differences in the treatment of these two cases. Furthermore, our results in this article are more general than those established for $G/K$ in  \cite{NS}.

In this article, we have limited our scope to the case $p\in [1, 2)\cup (2,\infty]$. This is  because, for $p=2$ the problem exhibits unique characteristics due to the one-dimensionality of the spectrum of Laplacian, which in fact makes it easier to handle, whereas the $p\neq 2$ cases can be approached in a coherent manner.

In Section \ref{Section 2}, we establish the notation and  gather all preliminary results we require. In Section \ref{Section 3}, we define multiplier and obtain some results on them. With this preparation, the main results in the general setting is proved in Section \ref{pcaseroestrichartz}. In the next section we apply them on weak $L^p$-functions and specific multipliers.

\section{Preliminaries}\label{Section 2}

\subsection{Generalities}
The letters $\mathbb{N}$, $\mathbb{Z}$, $\mathbb{R}$ and $\mathbb{C}$ will respectively denote the set of natural numbers, integers, real and complex numbers.  Set of all non-negative integers and nonzero complex numbers are denoted by $\mathbb{Z}_{+}$ and $\C^\times$ respectively. For $z\in \C$, $\Re z$ and $\Im z$  respectively are its  real and imaginary parts. If $f_{1}$ and $f_{2}$ are two non-negative functions on $X$, we say that $f_{1}\asymp f_{2}$ if there exist positive constants $A$ and $B$ such that $Af_{1}(t)\leq f_{2}(t)\leq Bf_{1}(t)$ for all $t$ in $X$. For two functions $f_{1}, f_{2}$ and distribution $T$, the notation $\langle f_{1}, f_{2}\rangle$ will mean $\int f_{1}f_{2}$ and  $\langle T, f_1\rangle$ will
indicate that  $T$ is acting on  $f_1$, whenever they make sense. For any $p\in(1,\infty)$, let $p^{\prime}$ denote the conjugate index $p/(p-1)$. We further assume $p^{\prime}=\infty$ when $p=1$ and vice-versa. For $p\in(1,\infty)$, let
$$\delta_{p}=\frac{1}{p}-\frac{1}{2}\quad\text{and}\quad S_{p}=\{z\in\mathbb{C}:|\Im z|\leq|\delta_{p}|\}.$$
We assume $\delta_{1}=-\delta_{\infty}=1/2$ so that $S_{1}=\{z\in\mathbb{C}:|\Im z|\leq 1/2\}$. The  interior and the boundary of $S_{p}$ will be denoted by  $S^{\circ}_{p}$ and $\partial S_{p}$ respectively.

\subsection{Weak \texorpdfstring{$L^p$}{lp} spaces}
For a measurable function $f:M\rightarrow\C$ on a $\sigma$-finite measure space $(M, m)$ and for  $p\in[1,\infty)$,  the weak $L^p$-norm of $f$ is denoted by $\|f\|_{p, \infty}$ and is defined by $\|f\|_{p, \infty} =\sup\limits_{t>0}td_{f}(t)^{1/p}$
where for $t>0$, $d_{f}(t)=m\left(\{x\in M: |f(x)|>t\}\right)$ is the distribution function of  $f$. The weak $L^p$ space  $L^{p,\infty}(M)$ is defined as the collection of all measurable functions $f:M\rightarrow\mathbb{C}$ which satisfy $\|f\|_{p, \infty}<\infty$. The space $L^{\infty,\infty}(M)$ is same as $L^{\infty}(M)$. It is known that $\|f\|_{p, \infty}\le \|f\|_p$ and hence $L^{p}(M)\subseteq L^{p,\infty}(M)$ (see \cite[Chapter 1]{LG1} for details).

\subsection{Homogeneous trees}
We briefly introduce the terminologies and definitions relevant to this article. We shall mainly follow \cite{CS2,FTP1,FTP2}. Our parametrization of various objects including the Poisson transform and the Helgason-Fourier transform will be consistent with that of \cite{CS2} and will differ slightly from those in \cite{FTP1,FTP2}.

Throughout this paper, $\mathcal{X}$ will denote a homogeneous tree of degree $q+1$ ($q\geq 2$), that is, a connected graph with no cycles where every vertex has $q+1$ neighbours. We identify $\mathcal{X}$ with the set of all vertices and endow $\mathcal{X}$ with the canonical discrete distance $d$ defined by the number of edges lying in the unique geodesic path joining the two vertices. We fix an arbitrary reference point $o$ in $\mathcal{X}$ and abbreviate $d(o,x)$ with $|x|$. We endow $\mathcal{X}$ with the counting measure and write $\#E$ to denote the cardinality of a finite subset $E$ of $\mathcal{X}$. For $x\in\mathcal{X}$ and $n\in\mathbb{Z}_{+}$, $B(x,n)$ and $S(x,n)$ will respectively denote the ball and the sphere centered at $x$ and of radius $n$. Observe that
$$\# S(x,n)=\begin{cases}
1 & \text{if }n=0,\\
(q+1)q^{n-1} & \text{if }n\neq 0,
\end{cases}\quad\text{and}\quad \# B(x,n)\asymp\# S(x,n)\asymp q^{n}.$$
Therefore the metric measure space $(\mathcal{X},d,\#)$ is of exponential volume growth. Let $G$ be the group of isometries of  $(\mathcal{X},d)$ and $K=\{g\in G:g(o)=o\}$. The group $G$ acts transitively on $\mathcal{X}$ via the map $(g,x)\mapsto g\cdot x=g(x)$. The action identifies $\mathcal{X}$ with the coset space $G/K$ so that functions on $\mathcal{X}$ identify to right-$K$-invariant functions on $G$ and vice-versa. We say that a function $f$ on $\mathcal{X}$ is radial if its value at the vertex $x$ depends only on $|x|$. Using the above identification, it follows that radial functions on $\mathcal{X}$ correspond to $K$-bi-invariant functions on $G$ and vice-versa. For any function space $E(\mathcal{X})$ of $\mathcal{X}$, $E(\mathcal{X})^{\#}$ will  denote the subspace of all radial functions in $E(\mathcal{X})$. For a complex-valued function $f$ on $\mathcal{X}$, its radialization $\mathscr{R}f$ is defined by
\begin{equation}\label{radializationoperator}
\mathscr{R} f(x)=\int\limits_{K}f(k\cdot x)dk,
\end{equation}
where $dk$ is the normalized Haar measure on $K$. Clearly, for any $f$, $\mathscr{R}f$ is  radial  and $\mathscr{R}f=f$ whenever $f$ is a radial. For two finitely supported functions $f, g$ on $\mathcal{X}$, with  $g$ radial, it is known that (see \cite[Chapter 3]{FTP2})
\begin{equation}\label{radializationproperty}
\langle\mathscr{R} f,g\rangle=\langle f,g\rangle
\end{equation}

The boundary $\Omega$ is defined as the set of all infinite geodesic rays, that is, chain of consecutively adjacent vertices of the form $\{\omega_{0},\omega_{1},\ldots\}$ starting at $\omega_{0}=o$. For $x\in\mathcal{X}$ and $\omega\in\Omega$, we define the confluence point $c(x,\omega)=x_{l}$, where $x_{l}$ is the last point lying on $\omega$ in the geodesic path $\{o,x_{1},\ldots,x_{n}\}$ connecting $o$ to $x$. For each $x\in\mathcal{X}$, the sectors
\begin{equation}\label{sectors}
E(x)=\{\omega\in\Omega:|c(x,\omega)|=|x|\},
\end{equation}
are compact and form a basis for the topology of $\Omega$. Let us consider the measure $\nu$ on $\Omega$ which assigns the mass $(q+1)^{-1}q^{-n+1}$ to each of the subsets $E(x)$, where $|x|=n\neq 0$. Then $\nu$ can be extended to all Borel subsets of $\Omega$, as a probability measure (see \cite[Chapter 3]{FTP2}).

\subsection{The Laplacian and the Poisson transform}\label{laplacianpoissontransform}
The Laplacian $\mathcal{L}$ on $\mathcal X$  is defined by the formula
\begin{equation}\label{laplaciandefinition}
\mathcal{L}f(x)=f(x)-\frac{1}{q+1}\sum\limits_{y:d(x,y)=1}f(y),\quad\text{for all }x\in\mathcal{X}.
\end{equation}

For $z\in\mathbb{C}$, the Poisson transform $\mathcal{P}_{z}F$ of a suitable function $F$ on $\Omega$ is defined by (\cite[p. 53]{FTP2})
$$\mathcal{P}_{z}F(x)=\int\limits_{\Omega}p^{1/2+iz}(x,\omega)F(\omega)d\nu(\omega),\quad\text{for all }x\in\mathcal{X},$$
where $p(x,\omega)$ denotes  the Poisson kernel which can be explicitly written as (see \cite[Chapter 3]{FTP2})
$$p(x,\omega)=q^{h_{\omega}(x)},\quad\text{for all }x\in\mathcal{X}\text{ and for all }\omega\in\Omega,$$
and $h_{\omega}(x)$ denotes the height function which is given by $h_{\omega}(x)=2|c(x,\omega)|-|x|$ (see \cite{CS2}). For a fixed $\omega\in\Omega$ and $z\in\mathbb{C}$, the function $x\mapsto p^{1/2+iz}(x,\omega)$ satisfies $\mathcal{L}p^{1/2+iz}(x,\omega)=\gamma(z)p^{1/2+iz}(x,\omega)$ and hence $\mathcal{L} (\mathcal{P}_{z}F)(x)= \gamma(z)\mathcal{P}_{z}F(x)$, for all $x\in\mathcal{X}$, where $\gamma:\mathbb{C}\rightarrow\mathbb{C}$ is an entire function  defined as
\begin{equation}\label{gammaz}
\gamma(z)=1-\frac{q^{1/2+iz}+q^{1/2-iz}}{q+1}.
\end{equation}

For $z\in\mathbb{C}$, the elementary spherical function $\phi_z$ is defined as the Poisson transform of the constant function $\bf{1}$. It is known that $\phi_{z}$ is a radial eigenfunction of $\mathcal{L}$ with eigenvalue $\gamma(z)$ and every radial eigenfunction of $\mathcal{L}$ with eigenvalue $\gamma(z)$ is a constant multiple of $\phi_{z}$ (see \cite[Theorem 1]{FTP1}). The explicit formula for $\phi_{z}$ is as follows (see \cite[Theorem 2]{FTP1} and \cite[Chapter 3, Theorem 2.2]{FTP2}):
\begin{equation}\label{eqsf}
\phi_z(x)= \begin{cases}
\vspace*{.2cm} \left(\frac{q-1}{q+1}|x|+1\right)q^{-|x|/2},&\text{for all }z\in\ \tau\mathbb{Z},\\
\vspace*{.2cm}\left(\frac{q-1}{q+1}|x|+1\right)q^{-|x|/2}(-1)^{|x|},&\text{for all }z\in {\tau/2}+\tau\mathbb{Z},\\
\mathbf{c}(z)q^{{(iz-1/2)}|x|}+\mathbf{c}(-z)q^{{(-iz-1/2)}|x|},&\text{for all }z\in\mathbb{C}\setminus(\tau/2)\mathbb{Z},
\end{cases}
\end{equation}
where $\mathbf{c}$ is the meromorphic function (aka the Harish-Chandra's c-function) given by
\begin{equation}\label{harishchandra}
\mathbf{c}(z)=\frac{q^{1/2}}{q+1}\frac{q^{1/2+iz}-q^{-{1/2}-iz}}{q^{iz}-q^{-iz}},\quad\text{for all }z\in\mathbb{C}\setminus(\tau/2)\mathbb{Z}.
\end{equation}
It is straightforward to see that for every $x\in\mathcal{X}$, the map $z\mapsto\phi_{z}(x)$ is an entire function and $|\phi_{z}(x)|\leq 1$ for all $z\in S_{1}$. From (\ref{eqsf}), it also follows that $\phi_{z}=\phi_{-z}=\phi_{z+\tau}$ for all $z\in\mathbb{C}$, where $\tau$ denotes the number $2\pi/\log q$.

\subsection{Fourier transforms on \texorpdfstring{$\mathcal{X}$}{fouriertransformonx}}
The spherical Fourier transform $\widehat{f}$ of a finitely supported radial function $f$ on $\mathcal{X}$ is defined by
$$\widehat{f}(z)=\sum\limits_{x\in\mathcal{X}}f(x)\phi_{z}(x),\quad\text{ where }z\in\mathbb{C}.$$
Since the map $z\mapsto\phi_{z}$ is even and $\tau$-periodic in $\mathbb{C}$, hence so is $\widehat{f}$. The Helgason-Fourier transform $\widetilde{f}$ of a finitely supported function $f$ on $\mathcal{X}$ is a function on $\mathbb{C}\times\Omega$ defined by the formula
$$\widetilde{f}(z,\omega)=\sum\limits_{x\in\mathcal{X}}f(x)p^{1/2+iz}(x,\omega).$$
Clearly $\widetilde{f}(z,\omega)=\widetilde{f}(z+\tau,\omega)$ for every $z\in\mathbb{C}$. A simple computation shows that if $f$ is radial, its Helgason-Fourier transform becomes independent of the variable $\omega$ and $\widetilde{f}(z,\omega)=\widehat{f}(z)$. Furthermore, for two finitely supported functions $f_{1}, f_2$ on $\mathcal{X}$, with  $f_{2}$ radial,
\begin{equation}\label{fourierconvolution}
\widetilde{(f_{1}\ast f_{2})}(z,\omega)=\widetilde{f}(z,\omega)\widehat{f_{2}}(z).
\end{equation}

\subsection{Schwartz spaces on \texorpdfstring{$\mathcal{X}$}{schwartzspacesonx}}
For $1\leq p<2$, let  $\mathcal{S}_p(\mathcal{X})$ be the space of all functions $\phi$ on $\mathcal{X}$ for which
\begin{equation}\label{schwartzspace}
\nu_{p,m}(\phi)=\sup\limits_{x\in\mathcal{X}}(1+|x|)^{m}q^{\frac{|x|}{p}}|\phi(x)|<\infty,\quad\text{for all }m\in\mathbb{Z}_{+}.
\end{equation}
It is known that $\mathcal{S}_p(\mathcal{X})$ forms a Fr\'{e}chet space with respect to these countable seminorms $\nu_{p,m}(\cdot)$ and they are also known as the $p$-Schwartz spaces of rapidly decreasing functions on $\mathcal{X}$ (see, for example \cite{CMS}). For $1\leq p< 2$, we  define $\mathcal{H}(S_{p})^{\#}$ as the space of all such functions $\psi:S_{p}\rightarrow\mathbb{C}$ which satisfies the following properties:
\begin{enumerate}
\item[(a)] $\psi$ is even and $\tau$-periodic on $S_{p}$.
\item[(b)] $\psi$ is analytic on $S_{p}^{\circ}$.
\item[(c)] $\psi$ and all its derivatives extend continuously to $\partial S_{p}$.
\item[(d)] For every non-negative integer $m$,
$$\mu_{p,m}(\psi)=\sup\limits_{z\in S_{p}}\left|\frac{d^m}{dz^{m}}\psi(z)\right|<\infty.$$
\end{enumerate}
We shall use the following isomorphism theorem of Schwartz spaces with its image.
\begin{Theorem}\label{lp isom}
The map $\phi\mapsto\widehat{\phi}$ is a topological isomorphism from $\mathcal{S}_p(\mathcal{X})^{\#}$ onto $\mathcal{H}(S_{p})^{\#}$, for every $p\in[1,2)$. In particular, for every $m\in\mathbb{Z}_{+}$ and for every $p\in[1,2)$, there exist positive constants $C_{1}$ and $C_{2}$ such that
$$ C_{1}\mu_{p,m}(\widehat{\phi})\leq\nu_{p,m+4}(\phi)\leq C_{2} \max\{\mu_{p,m+4}(\widehat{\phi}),\mu_{p,0}(\widehat{\phi})\},\quad\text{for all }\phi\in\mathcal{S}_p(\mathcal{X})^{\#}.$$
\end{Theorem}

\begin{proof}
For $p\in (1,2)$ it is proved in \cite[Theorem A.1]{R3} and for $p=1$ the proof is similar.
\end{proof}

\subsection{Tempered distributions on \texorpdfstring{$\mathcal{X}$}{tempereddistributionsonx}}
For $1\leq p< 2$, an $L^{p}$-tempered distribution on $\mathcal{X}$ is an element of the dual space $\mathcal{S}_p(\mathcal{X})^{\prime}$. More precisely, a linear functional $T:\mathcal{S}_p(\mathcal{X})\rightarrow\mathbb{C}$ is said to be an $L^{p}$-tempered distribution on $\mathcal{X}$ if $\langle T,\phi_{n}\rangle\rightarrow 0$ whenever $\nu_{p,m}(\phi_{n})\rightarrow 0$ as $n\rightarrow\infty$, for all $m\in\mathbb{Z}_{+}$. The distribution $T$ is said to be radial if
$$\langle T,\phi\rangle=\langle T,\mathscr{R} \phi\rangle,\quad\text{for all }\phi\in\mathcal{S}_p(\mathcal{X}).$$
The radialization of an $L^{p}$-tempered distribution $T$ is an $L^{p}$-tempered distribution defined by
$$\langle\mathscr{R}T,\phi\rangle=\langle T,\mathscr{R} \phi\rangle,\quad\text{for all }\phi\in\mathcal{S}_p(\mathcal{X}).$$
For a function $\phi$ on $\mathcal{X}$, its left translation by $g\in G$ is $\tau_{g}\phi(x)=\phi(g^{-1}\cdot x)$ and  $\tau_{g}$ of $T$ by an element $g\in G$ is defined as follows: If $\phi\in\mathcal{S}_p(\mathcal{X})$, then
$$\langle \tau_{g}T,\phi\rangle=T(\tau_{g^{-1}}\phi)=T\ast \phi^{\ast}(g^{-1}\cdot o),$$
where $\phi^{\ast}(g\cdot o)=\phi(g^{-1}\cdot o)$. Using Theorem \ref{lp isom}, the spherical Fourier transform $\widehat{T}$ of a radial $L^{p}$-tempered distribution $T$ is defined by the rule
$$\langle \widehat{T},\psi\rangle=\langle T,\phi\rangle,\text{ where }\psi\in\mathcal{H}(S_{p})^{\#}\text{ is such that }\widehat{\phi}=\psi.$$

\subsection{Characterization of Eigenfunctions as Poisson Transform}\label{characterizationsection}
We shall use the following result regarding the characterization of weak $L^{p}$-eigenfunctions of the Laplacian $\mathcal{L}$ as Poisson transform of $L^{p}$-functions defined on the boundary of $\mathcal{X}$.

\begin{Theorem}\label{weaklpchar}
Let $1\leq p<2$.  Suppose that $f$ is a complex-valued function on $\mathcal{X}$ and $z=\alpha+i\delta_{p^{\prime}}$, $\alpha\in\mathbb{R}$. Then $f=\mathcal{P}_{z}F$ for some $F\in L^{p^{\prime}}(\Omega)$ if and only if $f\in L^{p^{\prime},\infty}(\mathcal{X})$ and $\mathcal{L}f=\gamma(z)f$.
\end{Theorem}

\begin{proof}
When $p=1$, the result follows by putting $p=1$ and $r=\infty$ in Theorem \ref{hardypcase} (see Subsection \ref{remark on hardy type functions}), and further noticing that $H^{\infty}_{1}(\mathcal{X})=L^{\infty}(\mathcal{X})$ (since $\phi_{i\delta_{1}}(x)=1$ for all $x\in\mathcal{X}$). For $1<p<2$, we refer to \cite[Theorem B]{KR}.
\end{proof}


\section{Multiplier Operators on Schwartz Spaces}\label{Section 3}
Fix $1\leq p<2$. We denote by $Cv_{p}(\mathcal{X})$ the space of all radial functions $k$ on $\mathcal{X}$ for which the operator $\phi\mapsto \phi\ast k$ is bounded from $\mathcal{S}_p(\mathcal{X})$ to itself in the following sense: Corresponding to every semi-norm $\nu_{p,m_{2}}(\cdot)$ of $\mathcal{S}_p(\mathcal{X})$, there exists a semi-norm $\nu_{p,m_{1}}(\cdot)$ of $\mathcal{S}_p(\mathcal{X})$ and a constant $C_{m_{1},m_{2}}>0$ such that
\begin{equation}\label{multiplierboundedness1}
\nu_{p,m_{2}}(\phi\ast k)\leq C_{m_{1},m_{2}}~\nu_{p,m_{1}}(\phi),\quad\text{for all }\phi\in \mathcal{S}_p(\mathcal{X}).
\end{equation}
Further, we define $Cv_{p}^{\ast}(\mathcal{X})$ as the space of all radial functions $k$ on $\mathcal{X}$ for which the operator $\phi\mapsto \phi\ast k$ is bounded from $\mathcal{S}_p(\mathcal{X})$ to itself in the following sense: Corresponding to every semi-norm $\nu_{p,m}(\cdot)$ of $\mathcal{S}_p(\mathcal{X})$, there exists a constant $C_{m}>0$ such that
\begin{equation}\label{multiplierboundedness}
\nu_{p,m}(\phi\ast k)\leq C_{m}~\nu_{p,m}(\phi),\quad\text{for all }\phi\in \mathcal{S}_p(\mathcal{X}).
\end{equation}
Before getting into a formal definition of multipliers on Schwartz spaces, we investigate the following problems in this section:

\noindent{(a)} Firstly, we try to find a relation between the spaces $Cv_{p}(\mathcal{X})$ and $Cv_{p}^{\ast}(\mathcal{X})$.

\noindent{(b)} Secondly, we derive a necessary and sufficient condition on the function $k$ so that $k$ is in $Cv_{p}(\mathcal{X})$ and $Cv_{p}^{\ast}(\mathcal{X})$. In other words, we shall provide a characterization of the spaces $Cv_{p}(\mathcal{X})$ and $Cv_{p}^{\ast}(\mathcal{X})$ in terms of the behaviour of $k$.

\begin{Proposition}\label{lp_multiplier}
Let $1\leq p< 2$ and suppose that $k$ is a radial function on $\mathcal{X}$. Then the following are equivalent.
\begin{enumerate}
\item $k$ is in $\mathcal{S}_p(\mathcal{X})^{\#}$.
\item $k$ is in $Cv_{p}^{\ast}(\mathcal{X})$.
\item $k$ is in $Cv_{p}(\mathcal{X})$.
\end{enumerate}
\end{Proposition}

\begin{proof}
First we shall show that $(1)$ implies $(2)$. Since $k$ is radial, for a finitely supported function $\phi$  on $\mathcal{X}$ we have,
\begin{equation}\label{convolution}
\phi\ast k(x)=\sum\limits_{y\in\mathcal{X}}\phi(y)~k(d(x,y)),\quad\text{for all }x\in\mathcal{X}.
\end{equation}
Taking modulus on both sides of the expression above, using (\ref{schwartzspace}) and the inequality
$$1+d(o,x)\leq (1+d(o,y))(1+d(y,x)),\quad\text{for all }x,y\in\mathcal{X},$$
we get, for all $m\in\mathbb{Z}_{+}$,
\begin{equation}\label{convolution1}
|\phi\ast k(x)|\leq \nu_{p,m}(\phi)~\nu_{p,m+2}(k)~(1+d(o,x))^{-m} \mathlarger{\mathlarger{\mathlarger\sum}}\limits_{y\in\mathcal{X}}\frac{q^{-(d(o,y)+d(y,x))/p}}{(1+d(y,x))^{2}}.
\end{equation}
Let $c(x,y)$ be the confluence point of $x,y\in \mathcal X$, that is, the last point in common between the geodesic paths connecting $o$ to $x$ and $o$ to $y$. 
Since $d(o,y)=d(o,c(x,y))+d(c(x,y),y)$ and $d(x,y)=d(x,c(x,y))+d(c(x,y),y)$,  we have,
\begin{equation}\label{confluenceformula}
d(o,y)+d(x,y)=d(o,x)+2d(c(x,y),y),\quad\text{for all }y\in\mathcal{X}.
\end{equation}
Plugging in the formula above in (\ref{convolution1}), we obtain
\begin{equation} \label{new-1}
|\phi\ast k(x)|\leq \nu_{p,m}(\phi)~\nu_{p,m+2}(k)~q^{-d(o,x)/p}(1+d(o,x))^{-m} \mathlarger{\mathlarger{\mathlarger\sum}}\limits_{y\in\mathcal{X}}\frac{q^{-2d(c(x,y),y)/p}}{(1+d(y,x))^{2}}.
\end{equation}
Since $c(x,y)$ is a point on the geodesic $\{x_{0},x_{1},\ldots,x_{|x|}\}$ connecting $o$ to $x$, we can partition $\mathcal{X}$ as a disjoint union of the sets
\begin{equation}\label{gjx}
G_{j}(x)=\{y\in\mathcal{X}:c(x,y)=x_{j}\},\quad\text{for all }j\in\{0,1,\ldots,|x|\},
\end{equation}
so that
\begin{equation}\label{convolution2}
\mathlarger{\mathlarger{\mathlarger\sum}}\limits_{y\in\mathcal{X}}\frac{q^{-2d(c(x,y),y)/p}}{(1+d(y,x))^{2}}=
\mathlarger{\mathlarger{\mathlarger\sum}}\limits_{j=0}^{d(o,x)} \mathlarger{\mathlarger{\mathlarger\sum}}\limits_{y\in G_{j}(x)}\frac{q^{-2d(x_{j},y)/p}}{(1+d(y,x))^{2}}.
\end{equation}
Our next target is to determine all possible values of $d(x_{j},y)$ and $d(y,x)$ for each $y$ in $G_{j}(x)$. For this purpose we further decompose $G_{j}(x)$ as a disjoint union of the sets
\begin{equation}\label{gjnx}
G_{j,n}(x)=\{y\in G_{j}(x):d(x_{j},y)=n\},\quad\text{for all }n\in\mathbb{Z}_{+}.
\end{equation}
Then observe that $d(o,y)=n+j$, which together with (\ref{confluenceformula}) gives us
\begin{equation}\label{confluence2}
d(x,y)=d(o,x)+n-j,\quad\text{for all }y\in G_{j,n}(x).
\end{equation}
Furthermore, it is easy to see that $\# G_{j,0}(x)=1$ for all $j$, and if $n>0$,
\begin{equation}\label{confluence3}
\# G_{j,n}(x)=\begin{cases}
q^{n}, & \text{if }j=0\text{ and }d(o,x),\\
(q-1)q^{n-1}, & \text{otherwise}.
\end{cases}
\end{equation}
Consequently, decomposing $G_{j}(x)$ as a disjoint union of $G_{j,n}(x)$ and then using the formula (\ref{confluence2}) and the estimate (\ref{confluence3}), we get from (\ref{convolution2}),
$$\mathlarger{\mathlarger{\mathlarger\sum}}\limits_{y\in\mathcal{X}}\frac{q^{-2d(c(x,y),y)/p}}{(1+d(y,x))^{2}}=\mathlarger{\mathlarger{\mathlarger\sum}}\limits_{j=0}^{d(o,x)}\mathlarger{\mathlarger{\mathlarger\sum}}\limits_{n=0}^{\infty}\mathlarger{\mathlarger{\mathlarger\sum}}\limits_{y\in G_{j,n}(x)}\frac{q^{-2d(x_{j},y)/p}}{(1+d(y,x))^{2}}\leq\mathlarger{\mathlarger{\mathlarger\sum}}\limits_{j=0}^{d(o,x)}
\mathlarger{\mathlarger{\mathlarger\sum}}\limits_{n=0}^{\infty}\frac{q^{(1-2/p)n}}{(1+d(o,x)+n-j)^{2}}.$$
Interchanging the summation and substituting $j$ by $d(o,x)-j$ we get,
$$\mathlarger{\mathlarger{\mathlarger\sum}}\limits_{y\in\mathcal{X}}\frac{q^{-2d(c(x,y),y)/p}}{(1+d(y,x))^{2}}\leq\mathlarger{\mathlarger{\mathlarger\sum}}\limits_{n=0}^{\infty}\mathlarger{\mathlarger{\mathlarger\sum}}\limits_{j=0}^{d(o,x)}\frac{q^{(1-2/p)n}}{(1+j+n)^{2}}\leq \mathlarger{\mathlarger{\mathlarger\sum}}\limits_{n=0}^{\infty}q^{(1-2/p)n}\mathlarger{\mathlarger{\mathlarger\sum}}\limits_{j=0}^{\infty}\frac{1}{(1+j)^{2}}\leq C,$$
where we have used the fact that $p<2$. In view of \eqref{new-1}, this shows that $(1)$ implies $(2)$.

It is evident from the definition that $(2)$ implies $(3)$. To prove that $(3)$ implies $(1)$, we consider the Dirac measure $\delta_{o}$ at
the reference point $o$. It is easy to verify that $\delta_{o}\in \mathcal{S}_p(\mathcal{X})^{\#}$, $\nu_{p,m_{1}}(\delta_{o})=1$, for all $1\leq p< 2$, $m_{1}\in\mathbb{Z}_{+}$ and  $\delta_{o}\ast k=k$. Consequently, using (\ref{multiplierboundedness1}) we have $\nu_{p,m_{2}}(k)\leq C_{m_{1},m_{2}}$ for every $m_{2}\in\mathbb{Z}_{+}$, which proves $(1)$.
\end{proof}

For a fixed radial function $k$ on $\mathcal{X}$, the operator $\phi\mapsto \phi\ast k$ defined from $\mathcal{S}_{p}(\mathcal{X})$ to itself and satisfying (\ref{multiplierboundedness}) will henceforth be referred to as a ``multiplier on $\mathcal{S}_{p}(\mathcal{X})$ with symbol $\kappa(z):=\widehat{k}(z)$'', and will be denoted by $\Lambda$. Furthermore,  a multiplier $\Lambda$ with symbol $\kappa(z)$ satisfies the following properties: For any $\phi\in \mathcal{S}_{p}(\mathcal{X})$,

\noindent{(a)} $\widetilde{\Lambda\phi}(z,\omega)=\widetilde{\phi}(z,\omega)\kappa(z)$ for $z\in S_{p}$ and $\omega\in\Omega$ (see (\ref{fourierconvolution})),

\noindent{(b)} $\Lambda$ commutes with the left-translation $\tau_{g}$, i.e., $\tau_{g}(\Lambda\phi)=\Lambda(\tau_{g}\phi)$ for $g\in G$,

\noindent{(c)} $\Lambda$ commutes with the radialization operator $\mathscr{R}$.

Only (c) may need an elaboration. We first observe that if $\phi\in\mathcal{S}_{p}(\mathcal{X})$, by definition $\Lambda\phi=\phi\ast k\in\mathcal{S}_{p}(\mathcal{X})$ and hence $\mathscr{R}(\Lambda \phi)\in\mathcal{S}_{p}(\mathcal{X})^{\#}$. By Theorem \ref{lp isom}, $\widehat{\mathscr{R}(\Lambda\phi)}\in \mathcal{H}(S_{p})^{\#}$. Now using (\ref{radializationproperty}) we have, for all $z\in S_{p}$,
$$\widehat{\mathscr{R}(\Lambda\phi)}(z)=\langle \mathscr{R}(\Lambda\phi),\phi_{z} \rangle=\langle \phi\ast k,\phi_{z} \rangle=\langle \phi,\phi_{z}\ast k \rangle=\langle \mathscr{R}\phi, \phi_{z}\ast k \rangle = \langle \mathscr{R}\phi\ast k, \phi_{z} \rangle.$$
Hence $\widehat{\mathscr{R}(\Lambda\phi)}(z)=\widehat{\Lambda(\mathscr{R}\phi)}(z)$. Since $\Lambda(\mathscr{R}\phi)\in\mathcal{S}_{p}(\mathcal{X})^{\#}$, the assertion follows from Theorem \ref{lp isom}.

The action of a multiplier $\Lambda$ on $\mathcal{S}_{p}(\mathcal{X})$ extends naturally to the $L^{p}$-tempered distributions. If $T\in\mathcal{S}_p(\mathcal{X})^{\prime}$, $\Lambda T$ is an $L^{p}$-tempered distribution  on $\mathcal{X}$ defined by
$$\langle\Lambda T,\phi\rangle=\langle T,\Lambda\phi\rangle,\quad\text{for all }\phi\in\mathcal{S}_p(\mathcal{X}).$$

\begin{example} \label{example-multi}
Here we recall some examples of multipliers on $\mathcal{S}_{p}(\mathcal{X})$.

\noindent{(i)}  The Laplace operator $\mathcal{L}$ on $\mathcal{X}$  is given by right convolution with the function $(\delta_{o}-\mu_{1})\in \mathcal{S}_{p}(\mathcal{X})^{\#}$, where $\delta_{o}$ denotes the Dirac measure at the reference point $o$ and $\mu_{1}$ is the uniformly distributed probability measure supported on the sphere $S(o,1)$. Thus $\mathcal{L}$ is a multiplier on $\mathcal{S}_{p}(\mathcal{X})$ with symbol $\gamma(z)$ (see subsection \ref{laplacianpoissontransform}).

\noindent{(ii)} More generally, if $\Psi$ is a nonconstant holomorphic  function defined on a connected open set containing $\gamma(S_{p})$ then $\Psi\circ\gamma\in \mathcal{H}(S_{p})^{\#}$ for all $1\leq p< 2$. Hence by Proposition \ref{lp_multiplier}, $\Psi\circ\gamma$ corresponds to a  multiplier on $\mathcal{S}_{p}(\mathcal{X})$, which will be denoted by $\Psi(\mathcal{L})$. The complex time heat operator $e^{\xi\mathcal{L}}$, $\xi\in \C^\times$, ball and sphere averages are particular cases of $\Psi(\mathcal L)$ for such functions $\Psi$. We shall deal with them in Section \ref{refined}.

\noindent{(iii)} Since $\mathcal{S}_{r}(\mathcal{X})^{\#}\subseteq \mathcal{S}_{p}(\mathcal{X})^{\#}$ whenever $1\leq r\leq p<2$, for each $k\in \mathcal{S}_{r}(\mathcal{X})^{\#}$, the operator $\Lambda f=f\ast k$ is a multiplier on $\mathcal{S}_{p}(\mathcal{X})$ with symbol $\widehat{k}(z)$.

\noindent{(iv)} If $k\in L^{r}(\mathcal{X})^{\#}$ for $1\leq r<p<2$ then \cite[Theorem 4.1]{KR} states that $\widehat{k}(z)$ exists for all $z\in S^{\circ}_{r}$. Moreover, using Fubini’s theorem, Morera’s theorem and the fact that $z\mapsto\phi_{z}(x)$ is entire, we have $\widehat{k}\in \mathcal{H}(S_{p})^{\#}$. By Theorem \ref{lp isom} and Proposition \ref{lp_multiplier}, $\Lambda f=f\ast k$ is a multiplier on $\mathcal{S}_{p}(\mathcal{X})$ with symbol $\widehat{k}(z)$.
\end{example}

For future use we record here a result confirming that the operator $f\mapsto f\ast k$ is bounded from $L^{p^{\prime},\infty}(\mathcal{X})$ to itself where $k\in\mathcal{S}_p(\mathcal{X})^{\#}$, for a fixed $1\leq p< 2$.

\begin{Proposition}\label{multiplieronfunctionspaces}
For any fixed  $1\leq p< 2$ and $k\in\mathcal{S}_p(\mathcal{X})^{\#}$, there exists a seminorm $\nu_{p,m}(\cdot)$ of $\mathcal{S}_p(\mathcal{X})$ and a constant $C>0$ such that
$$\|f\ast k\|_{L^{p^{\prime}, \infty}(\mathcal{X})}\leq C~\nu_{p,m}(k) \|f\|_{L^{p^{\prime}, \infty}(\mathcal{X})},\quad\text{for all }f\in L^{p^{\prime},\infty}(\mathcal{X}).$$
Consequently, $L^{p^{\prime},\infty}(\mathcal{X})\ast\mathcal{S}_p(\mathcal{X})^{\#}\subseteq L^{p^{\prime},\infty}(\mathcal{X})$.
\end{Proposition}

\begin{proof}
We first consider the case $p=1$. If $k\in\mathcal{S}_1(\mathcal{X})^{\#}$, using (\ref{schwartzspace}) and the fact $\#S(o,n)\asymp q^{n}$, it follows that $\|k\|_{L^{1}(\mathcal{X})}\leq \nu_{1,2}(k)$. The desired result now follows by using Young's inequality. When $1<p<2$ and $k\in\mathcal{S}_p(\mathcal{X})^{\#}$, using \cite[Lemma 1]{P} we have $\|k\|_{L^{p,1}(\mathcal{X})}\leq C\nu_{p,2}(k)$. The proof now follows from the generalized Kunze-Stein phenomenon (see \cite{CMS}) and standard duality arguments.
\end{proof}

\section{Strichartz's theorem for multipliers}\label{pcaseroestrichartz}
In this section we consider analogues of Theorem \ref{th2} of  Strichartz for multipliers on $\mathcal{S}_{p}(\mathcal{X})$ for $1\leq p<2$. The following proposition is a crucial step towards this goal.

\begin{Proposition}\label{generalizedproposition}
Let $1\leq p< 2$. Let $\Lambda$ be a multiplier on $\mathcal{S}_{p}(\mathcal{X})$. Suppose that $T$ is an $L^{p}$-tempered distribution on $\mathcal{X}$ which is of the form
\begin{equation}\label{generalized1}
T=T_{1}+T_{2}+\cdots+T_{j},
\end{equation}
for some $L^{p}$-tempered distributions $T_{1},T_{2},\ldots,T_{j}$ on $\mathcal{X}$ that satisfy $(\Lambda-A_{i}I)^{N}T_{i}=0$ for all $1\leq i\leq j$ and for some $N\in\mathbb{N}$, where $A_{i}$'s are distinct complex numbers. Then
\begin{equation}\label{generalized2}
(\Lambda-A_{1}I)^{N}(\Lambda-A_{2}I)^{N}\cdots(\Lambda-A_{j}I)^{N}T=0.
\end{equation}
Conversely, if an $L^{p}$-tempered distribution $T$ satisfies (\ref{generalized2}) for some $N\in\mathbb{N}$, where $A_{i}$'s are distinct complex numbers, then there exist $L^{p}$-tempered distributions $T_{1},T_{2},\ldots,T_{j}$ on $\mathcal{X}$ such that $(\Lambda-A_{i}I)^{N}T_{i}=0$, for all $1\leq i\leq j$ and $T$ is of the form (\ref{generalized1}). Moreover, for $N=1$, the representation (\ref{generalized1}) is unique.
\end{Proposition}

\begin{proof}
Fix $1\leq p\leq 2$ and let $\Lambda$ be a multiplier on $\mathcal{S}_{p}(\mathcal{X})$. Assume that $T$ is of the form (\ref{generalized1}) for some $L^{p}$-tempered distributions $T_{1},T_{2},\ldots,T_{j}$ which satisfy the hypothesis of this lemma. Then for every $\phi\in \mathcal{S}_{p}(\mathcal{X})$,
\begin{align*}
\langle (\Lambda-A_{1}I)^{N}\cdots(\Lambda-A_{j}I)^{N}T,\phi \rangle&=\sum\limits_{i=1}^{j} \langle (\Lambda-A_{1}I)^{N}\cdots(\Lambda-A_{j}I)^{N}T_{i},\phi \rangle\\
&=\sum\limits_{i=1}^{j}\langle T_{i},(\Lambda-A_{i}I)^{N}\phi_{i}\rangle\\
&=\sum\limits_{i=1}^{j}\langle (\Lambda-A_{i}I)^{N}T_{i},\phi_{i}\rangle=0,
\end{align*}
where, for $1\leq i\leq j$, $\phi_{i}:=(\Lambda-A_{1}I)^{N}\cdots(\Lambda-A_{i-1}I)^{N}(\Lambda-A_{i+1}I)^{N}\cdots(\Lambda-A_{j}I)^{N}\phi$ is in $\mathcal{S}_{p}(\mathcal{X})$, since $\Lambda$ is a multiplier on $\mathcal{S}_{p}(\mathcal{X})$ (see Proposition \ref{lp_multiplier}).

Conversely, we assume that $T$ is an $L^{p}$-tempered distribution on $\mathcal{X}$ which satisfies (\ref{generalized2}) for some $N\in\mathbb{N}$, and for some distinct complex numbers $A_{1},A_{2},\ldots,A_{j}$. Our strategy is to construct, for each $i$, an operator $\mathbf{P_{i}}$ which acts as a projection from the space of all $L^{p}$-tempered distributions on $\mathcal{X}$ satisfying  (\ref{generalized2}) onto the generalized eigenspace $\{T^{\prime}\in\mathcal{S}_{p}(\mathcal{X})^{\prime}:(\Lambda-A_{i}I)^{N}T^{\prime}=0\}$, and satisfies
\begin{equation}\label{generalizedconstruction1}
\mathbf{P_{1}}+\mathbf{P_{2}}+\cdots+\mathbf{P_{j}}=\mathbf{I}.
\end{equation}
For this purpose, for each $i\in\{1,\ldots,j\}$, we seek for a polynomial of the form
$$P_{i}(z)=a_{0}+a_{1}z+\cdots+a_{jN-1}z^{jN-1},\quad\text{where }z\in\mathbb{C},$$
which satisfies the following conditions:
\begin{enumerate}
\item[$(\mathscr{C}1)$] For all $k\in\{1,\ldots,i-1,i+1,\ldots j\}$ and for all $m\in\{0,\ldots,N-1\}$, $P^{(m)}_{i}(A_{k})=0$, where $P^{(0)}_{i}(A_{k})=P_{i}(A_{k})$, and for $m\geq 1$, $P^{(m)}_{i}(A_{k})$ denotes the $m$-th derivative of $P_{i}$ at $A_{k}$.
\item[$(\mathscr{C}2)$] When $k=i$, $P_{i}(A_{i})=1$ and $P^{(m)}_{i}(A_{i})=0$ for all $m\in\{1,\ldots,N-1\}$.
\end{enumerate}
We note that conditions $(\mathscr{C}1)$ and $(\mathscr{C}2)$ give rise to a system of $jN$ linear equations with unknowns $a_{0},a_{1},\ldots,a_{jN-1}$. In fact, we can rewrite this system as
$$\begin{blockarray}{c(cccc)}
P_i(A_1)=0 & 1 & A_{1} & \cdots &A^{jN-1}_{1} \\
\vdots&\vdots&\vdots&\cdots& \vdots \\
P_i(A_i)=1 & 1 & A_{i} & \cdots &A^{jN-1}_{i}\\
\vdots&\vdots&\vdots&\cdots& \vdots \\
P_i(A_j)=0 & 1 & A_{j} & \cdots &A^{jN-1}_{j}\\[4mm]
\cline{2-5}\\
P^{(1)}_i(A_1)=0 & 0 & 1 & \cdots &(jN-1)A^{jN-2}_{1} \\
\vdots&\vdots&\vdots&\cdots& \vdots \\
P^{(1)}_i(A_j)=0 & 0 & 1 & \cdots & (jN-1)A^{jN-2}_{j}\\[4mm]
\cline{2-5}\\
\vdots & \vdots & \vdots &  & \vdots\\
\vdots & \vdots & \vdots & & \vdots\\[4mm]
\cline{2-5}\\
P^{(N-1)}_i(A_1)=0 & 0 & 0 & \cdots &(jN-1)\cdots(N(j-1)+1)A^{jN-N}_{1} \\
\vdots&\vdots&\vdots&\cdots& \vdots \\
P^{(N-1)}_i(A_j)=0 & 0 & 0 & \cdots & (jN-1)\cdots(N(j-1)+1)A^{jN-N}_{j}
\end{blockarray}~\left(\begin{array}{c}
a_{0}\\[18mm]
a_{1}\\[18mm]
\vdots\\[18mm]
\vdots\\[19mm]
a_{jN-1}
\end{array}\right)=\left(\begin{array}{c}
0\\
\vdots\\
1\\
\vdots\\
0\\[10mm]
0\\
\vdots\\
0\\[10mm]
\vdots\\
\vdots\\[10mm]
0\\
\vdots\\
0
\end{array}\right).$$
The above co-efficient matrix is also known as the confluent Vandermonde matrix which is non-singular, since the complex numbers $A_{i}$'s are distinct. See \cite[Section 5]{SA} for details. Consequently, we get a unique polynomial $P_{i}$ of degree less than or equal to $jN-1$ satisfying $(\mathscr{C}1)$ and $(\mathscr{C}2)$. On the other hand, by using conditions $(\mathscr{C}1)$ and $(\mathscr{C}2)$, it follows immediately that deg($P_{i}$)$\geq (j-1)N$ and hence, the polynomial $P_{i}$ can be uniquely written as
\begin{equation}\label{generalizedconstruction2}
P_{i}(z)=(z-A_{1})^{N}\cdots(z-A_{i-1})^{N}(z-A_{i+1})^{N}\cdots(z-A_{j})^{N}Q_{i}(z),
\end{equation}
where $Q_{i}$ is a unique polynomial of degree at most $N-1$ and
\begin{equation}\label{generalizedconstruction3}
Q_{i}(A_{i})=(A_{i}-A_{1})^{-N}\cdots(A_{i}-A_{i-1})^{-N}(A_{i}-A_{i+1})^{-N}\cdots(A_{i}-A_{j})^{-N}\neq 0.
\end{equation}
We claim that $P_{1}+P_{2}+\cdots+P_{j}=1$. To prove this, let us consider 
\begin{equation}\label{generalizedconstruction4}
P(z)=P_{1}(z)+P_{2}(z)+\cdots+P_{j}(z)-1.
\end{equation}
Then $P$ is a polynomial of degree at most $jN-1$. On the other hand, using conditions $(\mathscr{C}1)$ and $(\mathscr{C}2)$, it follows that $P^{(m)}(A_{i})=0$ for all $i\in\{1,\ldots,j\}$ and for all $m\in\{0,\ldots,N-1\}$. In other words, $P$ has $jN$ roots of the form $A_{1},A_{2},\ldots,A_{j}$, each of multiplicity $N$, which is only possible if $P$ is the zero polynomial. Now it is easy to see that the projection operator $\mathbf{P_{i}}$ defined by
\begin{equation}\label{projectionoperators}
\mathbf{P_{i}}=P_{i}(\Lambda)=(\Lambda-A_{1}I)^{N}\cdots(\Lambda-A_{i-1}I)^{N}(\Lambda-A_{i+1}I)^{N}\cdots(\Lambda-A_{j}I)^{N}Q_{i}(\Lambda),
\end{equation}
where $P_{i}$ and $Q_{i}$ is as in (\ref{generalizedconstruction2}) and (\ref{generalizedconstruction3}) respectively, satisfies (\ref{generalizedconstruction1}) (since $P$ defined by (\ref{generalizedconstruction4}) is a zero polynomial). Furthermore, we define
\begin{equation}\label{decompositionoftempered}
\langle T_{i},\phi\rangle=\langle\mathbf{P_{i}}T,\phi\rangle=\langle T,\mathbf{P_{i}}\phi\rangle,\quad\text{for each }i=1,\ldots,j,\text{ where }\phi\in\mathcal{S}_{p}(\mathcal{X}).
\end{equation}
Since $\Lambda$ is a multiplier on $\mathcal{S}_{p}(\mathcal{X})$, using Proposition \ref{lp_multiplier}, we have $\nu_{p,m}(\mathbf{P_{i}}\phi_{n})\rightarrow 0$ whenever $\nu_{p,m}(\phi_{n})\rightarrow 0$ as $n\rightarrow\infty$, for all $m\in\mathbb{Z}_{+}$ and for all $i$. This inturn implies that $T_{i}$ is an $L^{p}$-tempered distribution on $\mathcal{X}$. Combining (\ref{generalized2}), (\ref{projectionoperators}) and (\ref{decompositionoftempered}) it follows that $(\Lambda-A_{i}I)^{N}T_{i}=0$. Moreover, using (\ref{generalizedconstruction1}) and (\ref{decompositionoftempered}) we have $T=T_{1}+T_{2}+\cdots+T_{j}$.

Finally, we shall prove that if $T$ is an $L^{p}$-tempered distribution satisfying (\ref{generalized2}) for $N=1$, then $T$ can be uniquely decomposed in the form (\ref{generalized1}), where for each $i$, $\Lambda T_{i}=A_{i}T_{i}$. Seeking a contradiction, we assume that there exists two different sets of $L^{p}$-tempered distributions, namely $\{T_{i}:1\leq i\leq j\}$ and $\{T^{\prime}_{i}:1\leq i\leq j\}$ that satisfies the above properties. We can certainly assume that $T_{i}\neq T^{\prime}_{i}$ for all $i$. Then one sees immediately that
$$A_{1}^{k}(T_{1}-T^{\prime}_{1})+A_{2}^{k}(T_{2}-T^{\prime}_{2})+\cdots+A_{j}^{k}(T_{j}-T^{\prime}_{j})=0,\quad\text{for all } k=0,1,\ldots,j-1,$$
which results in a system of $j$ linear equations and the associated matrix
$$\left(\begin{array}{cccc}
1 & 1 & \cdots & 1\\
A_{1} & A_{2} & \cdots & A_{j}\\
\vdots & \vdots & \ddots & \vdots\\
A^{j-1}_{1} & A^{j-1}_{2} & \cdots & A^{j-1}_{j}
\end{array}\right)$$
is non-singular, as $A_{i}'$s are all distinct. Therefore $T_{i}-T^{\prime}_{i}=0$ for all $i$, which contradicts our assumption. This completes the proof.
\end{proof}

\begin{Remark}\label{remark-on-weak-lp}
Following are some observations which we record for future use:

\noindent{\textbf 1.} In Proposition \ref{generalizedproposition} if we assume that (\ref{generalized2}) holds for a weak $L^{p}$-function $f$ instead of a $L^{p}$-tempered distribution $T$ on $\mathcal{X}$ then we have the decomposition $f=f_{1}+\cdots+f_{j}$, for some weak $L^{p}$-functions $f_{i}$ satisfying $(\Lambda-A_{i}I)^{N}f_{i}=0$, for all $i=1,\ldots,j$. To see this, we first define
\begin{equation}\label{functiondecomposition}
f_{i}=\mathbf{P_{i}}f,\quad\text{for all }i=1,\ldots,j,
\end{equation}
where $\mathbf{P_{i}}$ is as in (\ref{projectionoperators}). Since $\mathbf{P_{i}}$ satisfies (\ref{generalizedconstruction1}), we have the desired decomposition. The fact that $(\Lambda-A_{i}I)^{N}f_{i}=0$ follows by combining (\ref{generalized2}) (with $f$ in place of $T$), (\ref{projectionoperators}) and (\ref{functiondecomposition}). Finally, $f_{i}\in L^{p^{\prime},\infty}(\mathcal{X})$ is a consequence of (\ref{functiondecomposition}) and Proposition \ref{multiplieronfunctionspaces}. Moreover, if (\ref{generalized2}) holds for $N=1$ then a step by step adaptation of the proof of Proposition \ref{generalizedproposition} implies that the decomposition $f=f_{1}+\cdots+f_{j}$ is unique.

\noindent{\textbf 2.} If in addition, we assume that the $L^{p}$-tempered distribution $T$ (in Proposition \ref{generalizedproposition}) and the weak $L^{p}$-function $f$ (in Remark \ref{remark-on-weak-lp} (1)) are radial, then $T_{i}$ and $f_{i}$ as in (\ref{decompositionoftempered}), (\ref{functiondecomposition}) respectively, are also radial.
\end{Remark}

We shall  also use of the following lemma which can be proved in much the same way as explained in \cite[p. 8]{HR2} and \cite[p. 130]{R3}. We omit its proof for brevity. 
\begin{Lemma}\label{generalizedtoeigenfunction}
Let $1\leq p<2$ and suppose that $\Lambda$ is a multiplier on $\mathcal{S}_{p}(\mathcal{X})$. Let $\{T_{k}\}_{k\in\mathbb{Z}_{+}}$ be an infinite sequence of radial $L^{p}$-tempered distributions on $\mathcal{X}$ satisfying, for all $k\in\mathbb{Z}_{+}$, $\Lambda T_{k}=A~T_{k+1}$ for some $A\in\C^\times$, and $|\langle T_{k},\phi\rangle|\leq M\nu_{p,m}(\phi)$ for a fixed semi-norm $\nu_{p,m}(\cdot)$ on $\mathcal{S}_{p}(\mathcal{X})$, for some $M>0$. If $(\Lambda-BI)^{N+1}T_{0}=0$ for some $N\in\mathbb{N}$ and for some $B\in\mathbb{C}^{\times}$ with $|A|=|B|$, then $(\Lambda-BI)T_{0}=0$.
\end{Lemma}

The remainder of this section will be devoted to the proof of our main results. Let $1\leq p<2$. We say that a multiplier $\Lambda$ on $\mathcal{S}_{p}(\mathcal{X})$ with symbol $\kappa$ is nontrivial if there exists some $z\in S_{p}$ such that $\kappa(z)\neq 0$.

\begin{Theorem}\label{maximumcase}
Let $1\leq p<2$. Let $\Lambda$ be a nontrivial multiplier on $\mathcal{S}_{p}(\mathcal{X})$  with symbol $\kappa(z)$. Suppose that  a bi-infinite sequence $\{T_{k}\}_{k\in\mathbb{Z}}$ of $L^{p}$-tempered distributions on $\mathcal{X}$ satisfies the following conditions:
\begin{enumerate}
\item[(i)] $\Lambda T_{k}=A~T_{k+1}$ for some $A\in\mathbb{C}^{\times}$ and for all $k\in \Z$.
\item[(ii)] For a fixed semi-norm $\nu_{p,m}(\cdot)$ of $\mathcal{S}_{p}(\mathcal{X})$ and a constant $M>0$, $|\langle T_{k},\phi\rangle|\leq M\nu_{p,m}(\phi)$ for all $k\in \Z$ and $\phi\in\mathcal{S}_{p}(\mathcal{X})$.
\end{enumerate}
Then the following assertions hold:
\begin{enumerate}
\item[(a)] If $|A|=\max\{|\kappa(z)|:z\in S_{p}\}$ and if the range of $\kappa$ intersects $\{w\in\mathbb{C}:|w|=|A|\}$ at finitely many distinct points $A_{1},\ldots,A_{j}$, then $T_{0}$ can be uniquely written as
$$T_{0}=T_{0,1}+T_{0,2}+\cdots+T_{0,j},$$
for some $L^{p}$-tempered distributions $T_{0,i}$ on $\mathcal{X}$ satisfying $\Lambda T_{0,i}=A_{i}T_{0,i}$ for all $i=1,\ldots,j$.
\item[(b)] If $|A|>\max\{|\kappa(z)|:z\in S_{p}\}$, then $T_{0}=0$.
\end{enumerate}
\end{Theorem}

\begin{proof}
{(a)} For convenience we divide the proof in a few steps.

\noindent{\textbf{Step 1:}} In this step we assume that  $T_{k}$ are radial distributions and prove that for a fixed $N\in\mathbb{Z}_{+}$,
\begin{equation}\label{lpgeneralized1}
(\Lambda-A_{1}I)^{N+1}(\Lambda-A_{2}I)^{N+1}\cdots(\Lambda-A_{j}I)^{N+1}T_{0}=0,
\end{equation}
which by Theorem \ref{lp isom} is equivalent  to showing that
\begin{equation}\label{lpgeneralized2}
(\kappa(z)-A_{1})^{N+1}(\kappa(z)-A_{2})^{N+1}\cdots(\kappa(z)-A_{j})^{N+1}\widehat{T}_{0}=0,
\end{equation}
where $A_{1},A_{2},\ldots,A_{j}$ are distinct complex numbers from the range of $\kappa$ such that for all $i$, $|A_{i}|=|A|=\max\{|\kappa(z)|:z\in S_{p}\}$. From hypothesis (i) of this theorem, we obtain $\Lambda^{k}T_{-k}=A^{k}T_{0}$ for all $k\in\mathbb{Z}_{+}$. Taking spherical Fourier transform of both sides, we get
\begin{equation}\label{recurrence1}
\kappa(z)^{k}\widehat{T}_{-k}=A^{k}\widehat{T}_{0},\quad\text{for all }k\in\mathbb{Z}_{+}.
\end{equation}
We recall that $z\mapsto\kappa(z)$ is nontrivial and hence so is $A$. This fact together with (\ref{recurrence1}) and hypothesis (ii) yields, for all $\psi\in\mathcal{H}(S_{p})^{\#}$,
\begin{align*}
|\langle(\kappa(z)-A_{1})^{N+1}\cdots(\kappa(z)-A_{j})^{N+1}\widehat{T}_{0},\psi\rangle|&\\
&\hspace*{-0.8in}=\left|\left\langle\widehat{T}_{-k},\left(\frac{\kappa(z)}{A}\right)^{k}(\kappa(z)-A_{1})^{N+1}\cdots(\kappa(z)-A_{j})^{N+1}\psi\right\rangle\right|\\
&\hspace*{-0.8in}=\left|\left\langle T_{-k},\left(\left(\frac{\kappa(z)}{A}\right)^{k}(\kappa(z)-A_{1})^{N+1}\cdots(\kappa(z)-A_{j})^{N+1}\psi\right)^{\vee}\right\rangle\right|\\
&\hspace*{-0.85in}\leq M\nu_{p,m}\left[\left(\left(\frac{\kappa(z)}{A}\right)^{k}(\kappa(z)-A_{1})^{N+1}\cdots(\kappa(z)-A_{j})^{N+1}\psi\right)^{\vee}\right].
\end{align*}
Thus by Theorem \ref{lp isom}, there exists $m\in\mathbb{Z}_{+}$ such that
\begin{multline*}
|\langle(\kappa(z)-A_{1})^{N+1}\cdots(\kappa(z)-A_{j})^{N+1}\widehat{T}_{0},\psi\rangle|\\
\leq M\max\limits_{n=0,m}\left\{\mu_{p,n}\left[\left(\frac{\kappa(z)}{A}\right)^{k}(\kappa(z)-A_{1})^{N+1}\cdots(\kappa(z)-A_{j})^{N+1}\psi\right]\right\}.
\end{multline*}
We choose $N=5m+1$. To prove (\ref{lpgeneralized2}), it is enough to show that
$$\mu_{p,n}\left[\left(\frac{\kappa(z)}{A}\right)^{k}(\kappa(z)-A_{1})^{5m+2}\cdots(\kappa(z)-A_{j})^{5m+2}\psi\right]\rightarrow 0\quad\text{as }k\rightarrow\infty,\text{ for each }n=0,m.$$
We first consider the case $n=m$. Since $z\mapsto\kappa(z)$ and $z\mapsto\psi(z)$ are even and $\tau$-periodic functions on $S_{p}$, we have
$$\mu_{p,m}\left[\left(\frac{\kappa(z)}{A}\right)^{k}(\kappa(z)-A_{1})^{5m+2}\cdots(\kappa(z)-A_{j})^{5m+2}\psi\right]=\sup\limits_{z\in S^{+}_{p}} F_{k}(z),$$
where $S^{+}_{p}=\{z\in S_{p}: |\Re z|\leq \tau/2\text{ and }\Im z\geq 0\}$ and for $z\in S^{+}_{p}$,
$$F_{k}(z)=\left|\frac{d^{m}}{dz^{m}}\left(\left(\frac{\kappa(z)}{A}\right)^{k}(\kappa(z)-A_{1})^{5m+2}\cdots(\kappa(z)-A_{j})^{5m+2}\psi(z)\right)\right|.$$
Hence, it suffices to prove that
\begin{equation}\label{step1mainresult}
\sup\limits_{z\in S^{+}_{p}} F_{k}(z)\rightarrow 0\quad\text{as}\quad k\rightarrow\infty.
\end{equation}
For this purpose, we shall decompose $S^{+}_{p}$ into finitely many disjoint subsets and compute the supremum of $F_{k}(z)$ over each of them. Before going into the details, let us first observe that for all $z\in S^{+}_{p}$ and for all $i=1,\ldots,j$,
\begin{align}
F_{k}(z)&\leq B_{m}\sum\limits_{a+b+c=m}\left|\frac{d^{a}}{dz^{a}}\left(\left(\frac{\kappa(z)}{A}\right)^{k}\right)\right|\times\left|\frac{d^{b}}{dz^{b}}\left((\kappa(z)-A_{i})^{5m+2}\right)\right|\nonumber\\
&\hspace*{1.7in}\times\left|\frac{d^{c}}{dz^{c}}\left((\kappa(z)-A_{1})^{5m+2}\cdots(\kappa(z)-A_{i-1})^{5m+2}\right.\right.\nonumber\\
&\hspace*{2.65in}\left.\cdot~(\kappa(z)-A_{i+1})^{5m+2}\cdots(\kappa(z)-A_{j})^{5m+2}\psi(z)\right)\Big|\nonumber\\
&\leq C_{m}~k^{m}~\left|\frac{\kappa(z)}{A}\right|^{k-m}~|\kappa(z)-A_{i}|^{4m+2}\label{pcase2ndinequality}\\
&\leq D_{m}~k^{m}~\left|\frac{\kappa(z)}{A}\right|^{k-m}.\label{pcase1stinequality}
\end{align}

Since  $A_{1},\ldots,A_{j}\in \C$ are distinct, we can find a small $r>0$ such that the closed balls $B(A_{i},r)=\{w\in\mathbb{C}:|w-A_{i}|\leq r\}$, $i=1,\ldots,j$ are all pairwise disjoint. We define the sets
\begin{align*}
U_{i}&=\{z\in S^{+}_{p}:\kappa(z)\in B(A_{i},r)\},\quad\text{for }i=1,\ldots,j,\quad\text{and}\\
V&=S^{+}_{p}\setminus\left(~\bigcup\limits_{i=1}^{j}U_{i}~\right).
\end{align*}
Note that the closure  $\overline{V}$ of $V$ is a compact subset of $S^{+}_{p}$ and $\kappa(\overline{V})$ does not contain the points $A_{1},\ldots,A_{j}$. Furthermore, the map $z\mapsto |\kappa(z)/A|$ is continuous on $S^{+}_{p}$. Therefore
$$\sup\limits_{z\in\overline{V}}\left|\frac{\kappa(z)}{A}\right|=\alpha<1.$$
From this and  (\ref{pcase1stinequality}), we finally get
\begin{equation}\label{pcase1estimate}
\sup\limits_{z\in V} F_{k}(z)\leq C_{m} k^{m}\alpha^{k-m}\rightarrow 0\quad\text{as}\quad k\rightarrow\infty.
\end{equation}

We next turn to estimating the supremum of $F_{k}(z)$ over $U_{i}$. Fix $i\in\{1,\ldots,j\}$. The fact $|A_{i}|=|A|\neq 0$ together with the estimate (\ref{pcase2ndinequality}) yields
\begin{equation}\label{p2ndinequality}
F_{k}(z)\leq E_{m}~k^{m}~\left|\frac{\kappa(z)}{A_{i}}\right|^{k-m}\left|\frac{\kappa(z)}{A_{i}}-1\right|^{4m+2},\quad\text{for all }z\in U_{i}.
\end{equation}
It is easy to see that $U_{i}$ is a compact set in $S^{+}_{p}$ whose range under the map $\kappa$ contains $A_{l}$ only when $l=i$. This, together with the continuity of $\kappa$ and the fact $|A_{i}|=\sup\{|\kappa(z)|:z\in S_{p}\}$ implies that the image of $U_{i}$ under the map $z\mapsto \kappa(z)/A_{i}$ is a compact subset of $\mathbb{D}\cup\{1\}$, where $\mathbb{D}=\{w\in\mathbb{C}:|w|<1\}$. Now, consider a one-parameter family of sets of the form $\{\mathbb{D}_{s}:0<s<1\}$, where $\mathbb{D}_{s}=\{w\in\mathbb{C}:(\Re w-(1-s))^{2}+(\Im w)^{2}\leq s^{2}\}$. Clearly, $\{\mathbb{D}_{s}:0<s<1\}$ is an expanding collection of sets that covers $\mathbb{D}\cup\{1\}$. From the preceding argument, there exists a fixed $s_{i}\in(0,1)$ such that $\kappa(U_{i})/A_{i}\subset \mathbb{D}_{s_{i}}$ and $0\in\mathbb{D}_{s_{i}}$. Therefore, in view of (\ref{p2ndinequality}) and (\ref{step1mainresult}) it is only left to prove that
\begin{equation}\label{discinequality}
\sup\limits_{w\in\mathbb{D}_{s_{i}}} k^{m}~|w|^{k-m}~|w-1|^{4m+2}\rightarrow 0\quad\text{as}\quad k\rightarrow\infty.
\end{equation}

To this end, we decompose $\mathbb{D}_{s_{i}}$ into the sets $V_{i,k}$ and $V^{\mathrm{C}}_{i,k}=\mathbb{D}_{s_{i}}\setminus V_{i,k}$, where for all $k\geq 4$,
$$V_{i,k}=\{w\in\mathbb{C}:\Re w=(1-s_{i})+s\cos\theta,\Im w=s\sin\theta,~s_{i}-k^{-1/2}\leq s\leq s_{i},~-k^{-1/4}\leq\theta\leq k^{-1/4}\}.$$
Clearly, $V_{i,k}$ is a compact subset of $\mathbb{D}_{s_{i}}$ that contains the point $1$. Furthermore, for every $w\in V_{i,k}$,
$$|1-w|^{2}=(s_{i}-s\cos\theta)^{2}+s^{2}\sin^{2}\theta\leq \left(s_{i}-\left(s_{i}-\frac{1}{\sqrt{k}}\right)\cos\left(\frac{1}{k^{1/4}}\right)\right)^{2}+s_{i}^{2}\sin^{2}\left(\frac{1}{k^{1/4}}\right),$$
which together with the trivial estimates $(|x|^{2}+|y|^{2})^{1/2}\leq |x|+|y|$ for all $x,y$ real, $\sin\theta\leq \theta$ whenever $\theta\in[0,1]$ and the fact $0<s_{i}<1$ yields
\begin{align*}
|1-w|&\leq\left|\left(s_{i}-\left(s_{i}-\frac{1}{\sqrt{k}}\right)\right)+\left(s_{i}-\frac{1}{\sqrt{k}}\right)\left(1-\cos\left(\frac{1}{k^{1/4}}\right)\right)\right|+\frac{s_{i}}{k^{1/4}}\\
&\leq\frac{1}{\sqrt{k}}+2\left(s_{i}-\frac{1}{\sqrt{k}}\right)\sin^{2}\left(\frac{1}{2k^{1/4}}\right)+\frac{s_{i}}{k^{1/4}}\\
&\leq \frac{4}{k^{1/4}}.
\end{align*}
Consequently, there exists a constant $C_{1}>0$ such that
\begin{equation}\label{disc1stestimate}
k^{m}~|w|^{k-m}~|w-1|^{4m+2}\leq C_{1}~k^{-1/2},\quad\text{for all }k\geq 4\text{ and for all }w\in V_{i,k}.
\end{equation}

Our next aim is to compute the supremum of the quantity on the left hand side of (\ref{disc1stestimate}) over the set $V^{\mathrm{C}}_{i,k}$. We first prove that there exists a constant $C_{2}>0$ such that
\begin{equation}\label{disc2ndestimate}
|w|\leq \left(1+\frac{C_{2}}{\sqrt{k}}\right)^{-1/2},\quad\text{for all }w\in V^{\mathrm{C}}_{i,k}.
\end{equation}
Fix $w\in V^{\mathrm{C}}_{i,k}$ and suppose that $\Re w=(1-s_{i})+s\cos\theta$ and $\Im w=s\sin\theta$. Depending on the possibilities, we consider the following two parts. Firstly, let $0\leq s< s_{i}-1/\sqrt{k}$ and $-\pi\leq\theta\leq\pi$. Then $w$ is non-zero. Moreover, an easy computation shows that
\begin{equation}\label{vcompliment1}
1-|w|\geq 1-\left[(1-s_{i})^{2}+2(1-s_{i})\left(s_{i}-\frac{1}{\sqrt{k}}\right)+\left(s_{i}-\frac{1}{\sqrt{k}}\right)^{2}\right]^{1/2}=\frac{1}{\sqrt{k}}.
\end{equation}
Secondly, if $s_{i}-1/\sqrt{k}\leq s\leq s_{i}$ and $\theta\in[-1/k^{1/4},1/k^{1/4}]^{\mathrm{C}}$, then
\begin{equation}\label{vcompliment2}
1-|w|^{2}\geq 1-\left[(1-s_{i})^{2}+s_{i}^{2}+2s_{i}(1-s_{i})\cos\left(\frac{1}{k^{1/4}}\right)\right]=4s_{i}(1-s_{i})\sin^{2}\left(\frac{1}{2k^{1/4}}\right)\geq\frac{C_{i}}{\sqrt{k}},
\end{equation}
for some constant $C_{i}$ which depends only on $s_{i}$. Now dividing both sides of (\ref{vcompliment1}) and (\ref{vcompliment2}) by $|w|$ and $|w|^{2}$ (for $w\neq 0$) respectively, we get (\ref{disc2ndestimate}) with $C_{2}=\min\{1,C_{i}\}$. Consequently,
\begin{equation}\label{disc3rdestimate}
k^{m}~|w|^{k-m}~|w-1|^{4m+2}\leq C_{3}~k^{m}~\left(1+\frac{C_{2}}{\sqrt{k}}\right)^{(-k+m)/2},\quad\text{for all }w\in V^{\mathrm{C}}_{i,k}.
\end{equation}
Letting $k\rightarrow\infty$ in (\ref{disc1stestimate}) and (\ref{disc3rdestimate}), we get (\ref{discinequality}), which combined with (\ref{pcase1estimate}) implies (\ref{step1mainresult}) and finally proves (\ref{lpgeneralized1}).

\noindent{\textbf{Step 2:}} Keeping the radiality of the sequence $\{T_{k}\}_{k\in\mathbb{Z}}$ intact, here we shall show that  $N=0$ in (\ref{lpgeneralized1}). Let us assume, on the contrary, that
\begin{equation}\label{contradictionexpression}
(\Lambda-A_{1}I)(\Lambda-A_{2}I)\cdots(\Lambda-A_{j}I)T_{0}\neq 0.
\end{equation}
In view of (\ref{lpgeneralized1}) and (\ref{contradictionexpression}), it is possible to choose a natural number $N_{0}$ satisfying $1\leq N_{0}\leq N$ for which
\begin{align}
(\Lambda-A_{1}I)^{N_{0}}(\Lambda-A_{2}I)^{N_{0}}\cdots(\Lambda-A_{j}I)^{N_{0}}T_{0}&\neq 0,\quad\text{and}\label{contradictionexpression2}\\
(\Lambda-A_{1}I)^{N_{0}+1}(\Lambda-A_{2}I)^{N_{0}+1}\cdots(\Lambda-A_{j}I)^{N_{0}+1}T_{0}&=0.\label{usinggeneralisedproposition}
\end{align}
Then Proposition \ref{generalizedproposition} together with (\ref{usinggeneralisedproposition}) gives us the decomposition
\begin{equation}\label{decomposition2}
T_{0}=T_{0,1}+T_{0,2}+\cdots+T_{0,j},
\end{equation}
for some $L^{p}$-tempered distributions $T_{0,1},\ldots,T_{0,j}$ on $\mathcal{X}$ satisfying $(\Lambda-A_{i}I)^{N_{0}+1}T_{0,i}=0$ for all $i=1,\ldots,j$. Indeed, using (\ref{projectionoperators}) and (\ref{decompositionoftempered}), we have the explicit expression $T_{0,i}=\mathbf{P_{i}}T_{0}$ for each $i$, where $\mathbf{P_{i}}$ is a projection operator of the form
\begin{equation}\label{projectionoperators2}
\mathbf{P_{i}}=(\Lambda-A_{1}I)^{N_{0}+1}\cdots(\Lambda-A_{i-1}I)^{N_{0}+1}(\Lambda-A_{i+1}I)^{N_{0}+1}\cdots(\Lambda-A_{j}I)^{N_{0}+1}Q_{i}(\Lambda),
\end{equation}
and $Q_{i}$ is a nonzero polynomial of degree at most $N_{0}$. Furthermore, combining (\ref{contradictionexpression2}) and (\ref{decomposition2}), it follows that for a fixed $i_{0}\in\{1,\ldots,j\}$,
\begin{equation}\label{contradiction}
(\Lambda-A_{i_{0}}I)^{N_{0}}T_{0,i_{0}}\neq 0.
\end{equation}
Now, using (\ref{projectionoperators2}), let us define
$$T_{k,i_{0}}=\mathbf{P_{i_{0}}}T_{k},\quad\text{for all }k\in\mathbb{Z}_{+}.$$
Observe that $\{T_{k,i_{0}}\}_{k\in\mathbb{Z}_{+}}$ is an infinite sequence of radial $L^{p}$-tempered distributions on $\mathcal{X}$. Indeed, the radiality of the sequence $\{T_{k}\}_{k\in\mathbb{Z}_{+}}$ and the fact that $\mathscr{R}(\Lambda\phi)=\Lambda(\mathscr{R}\phi)$ for all $\phi\in\mathcal{S}_{p}(\mathcal{X})$, together yields
$$\langle T_{k,i_{0}},\mathscr{R}\phi\rangle=\langle \mathbf{P_{i_{0}}}T_{k},\mathscr{R}\phi\rangle=\langle T_{k},\mathscr{R}(\mathbf{P_{i_{0}}}\phi)\rangle=\langle \mathscr{R}T_{k},\mathbf{P_{i_{0}}}\phi\rangle=\langle T_{k},\mathbf{P_{i_{0}}}\phi\rangle=\langle T_{k,i_{0}},\phi\rangle.$$
It also follows from hypothesis (i) of this theorem that $\Lambda T_{k,i_{0}}=A~ T_{k+1,i_{0}}$ for all $k\in\mathbb{Z}_{+}$. Furthermore, using Proposition \ref{lp_multiplier} and hypothesis (ii) of this theorem, we get
$$|\langle T_{k,i_{0}},\phi\rangle|=|\langle T_{k},\mathbf{P_{i_{0}}}\phi\rangle|\leq M~\nu_{p,m}(\mathbf{P_{i_{0}}}\phi)\leq M_{1}~\nu_{p,m}(\phi),\text{ for all }\phi\in\mathcal{S}_{p}(\mathcal{X})\text{ and for all }k\in\mathbb{Z}_{+}.$$
Therefore, the infinite sequence $\{T_{k,i_{0}}\}_{k\in\mathbb{Z}_{+}}$ satisfies all the hypothesis of Lemma \ref{generalizedtoeigenfunction} and $(\Lambda-A_{i_{0}}I)^{N_{0}+1}T_{0,i_{0}}=0$ for some $N_{0}\in\mathbb{N}$, with $|A_{i_{0}}|=|A|\neq 0$. Consequently, by Lemma \ref{generalizedtoeigenfunction}, $(\Lambda-A_{i_{0}}I)T_{0,i_{0}}=0$, which is a contradiction to the expression (\ref{contradiction}) and our assumption on $N_{0}$. Hence $N=0$ in (\ref{lpgeneralized1}), that is,
\begin{equation}\label{radiallpdistribution}
(\Lambda-A_{1}I)(\Lambda-A_{2}I)\cdots(\Lambda-A_{j}I)T_{0}=0.
\end{equation}
Applying Proposition \ref{generalizedproposition} to the expression above, we finally get the desired conclusion for radial $L^{p}$-tempered distributions.

\noindent{\textbf{Step 3:}} We shall  withdraw the  assumption of radiality on  $\{T_{k}\}_{k\in\mathbb{Z}}$. For each $g$ in $G$, we define the sequence $\{\mathscr{R}(\tau_{g}T_{k})\}_{k\in\mathbb{Z}}$ of radial $L^{p}$-tempered distributions on $\mathfrak{X}$.
Then using hypothesis (i)  together with the fact that $\Lambda$ commutes with $\tau_g$ and $\mathscr R$, we get
$$\Lambda(\mathscr{R}(\tau_{g}T_{k}))=A~\mathscr{R}(\tau_{g}T_{k+1}),\quad\text{for all }k\in\mathbb{Z}.$$
Furthermore, using hypothesis (ii), we obtain, for all $\phi\in\mathcal{S}_{p}(\mathfrak{X})^{\#}$,
\begin{equation}\label{seminormtranslation}
|\langle\mathscr{R}(\tau_{g}T_{k}),\phi\rangle|=|\langle\tau_{g}T_{k},\phi\rangle|=|\langle T_{k},\tau_{g^{-1}}\phi\rangle|\leq M~\nu(\tau_{g^{-1}}\phi).
\end{equation}
 Since,  $|g^{-1}\cdot y|\leq (|g^{-1}\cdot o|+|y|)$ and $(1+|g^{-1}\cdot y|)\leq (1+|g^{-1}\cdot o|)(1+|y|)$ for all $g\in G$ and  $y \in \mathfrak{X}$, we have (see \eqref{schwartzspace})
\begin{align*}
\nu_{p,m}(\tau_{g^{-1}}\phi)&=\sup\limits_{y\in\mathfrak{X}}(1+|y|)^{m}q^{|y|/p}|\phi(g\cdot y)|\\
&=\sup\limits_{y\in\mathfrak{X}}(1+|g^{-1}\cdot y|)^{m}q^{|g^{-1}\cdot y|/p}|\phi(y)|\\
&\leq(1+|g^{-1}\cdot o|)^{m}q^{|g^{-1}\cdot o|/p}\sup\limits_{y\in\mathfrak{X}}(1+|y|)^{m}q^{|y|/p}|\phi(y)|\\
&=M_{g}~\nu_{p,m}(\phi),
\end{align*}
where $M_{g}$ is a constant depending only on $g$. Combining this  with (\ref{seminormtranslation}), we get
$$|\langle\mathscr{R}(\tau_{g}T_{k}),\phi\rangle|\leq M_{g}~\nu(\phi),\quad\text{for all }k\in\mathbb{Z}\text{ and for all }\phi\in\mathcal{S}_{p}(\mathfrak{X})^{\#}.$$
Thus for each $g \in G$, the sequence $\{\mathscr{R}(\tau_{g}T_{k})\}_{k\in\mathbb{Z}}$  of radial $L^{p}$-tempered distributions satisfies the hypotheses of this theorem. Hence by Step 2, we have $(\Lambda-A_{1}I)(\Lambda-A_{2}I)\cdots(\Lambda-A_{j}I)\mathscr{R}(\tau_{g}T_{0})=0$, for all $g$ in $G$. Using  commutativity of $\Lambda$ with $\tau_g$ and $\mathscr R$ we conclude that
$$\mathscr{R}(\tau_{g}((\Lambda-A_{1}I)(\Lambda-A_{2}I)\cdots(\Lambda-A_{j}I)T_{0}))=0,\quad\text{for all }g\in G.$$
This implies that $(\Lambda-A_{1}I)(\Lambda-A_{2}I)\cdots(\Lambda-A_{j}I)T_{0}=0$.
Indeed, if for an $L^{p}$-tempered distribution $T$, if $\mathscr{R}(\tau_{g}T)=0$ for every $g\in G$, then
$$\langle\tau_{g}T,\delta_{o}\rangle=T\ast\delta_{o}(g^{-1}\cdot o)=0,\quad\text{for all }g\in G,$$
where $\delta_{o}$ denotes the Dirac measure at the reference point $o$. Since $T\ast\delta_{o}=T$ in the sense of distribution, we have $T=0$.

The assertion follows from this and  Proposition \ref{generalizedproposition}.

\noindent{\textit{Proof of Part \rm{(b)}:}} Assuming the radiality of the distributions $T_{k}$ and using the recurrence relation (\ref{recurrence1}), for a fixed $\psi\in\mathcal{H}(S_{p})^{\#}$, we get
$$|\langle\widehat{T}_{0},\psi\rangle|=\left|\left\langle\widehat{T}_{-k},\left(\frac{\kappa(z)}{A}\right)^{k}\psi\right\rangle\right|=\left|\left\langle T_{-k},\left(\left(\frac{\kappa(z)}{A}\right)^{k}\psi\right)^{\vee}\right\rangle\right|,$$
which together with hypothesis (ii) of this theorem and  Theorem \ref{lp isom} yields
$$|\langle\widehat{T}_{0},\psi\rangle|\leq M \max\limits_{n=0,m}\left\{\mu_{p,n}\left[\left(\frac{\kappa(z)}{A}\right)^{k}\psi\right]\right\},\quad\text{for a fixed }m\in\mathbb{N}.$$
Clearly, the given condition $|A|>\max\{|\kappa(z)|:z\in S_{p}\}$ implies that
$$\mu_{p,n}\left[\left(\frac{\kappa(z)}{A}\right)^{k}\psi\right]\rightarrow 0\quad\text{as}\quad k\rightarrow\infty,\quad\text{for each }n=0,m.$$
Consequently, $\widehat{T}_{0}=0$ and hence so is $T_{0}$. The general case now follows by arguing in a similar way as in part (a) of this theorem.
\end{proof}

\begin{Theorem}\label{minimumcase}
Let $1\leq p<2$. Let $\Lambda$ be a multiplier on $\mathcal{S}_{p}(\mathcal{X})$ associated with symbol $\kappa(z)$ satisfying $\kappa(z)\neq 0$ for all $z\in S_{p}$. Suppose that the sequence $\{T_k\}$ in Theorem \ref{maximumcase} is infinite (instead of bi-infinite) and parametrized by $\Z_+$ and they satisfy conditions (i) and (ii) of  Theorem \ref{maximumcase} for all $k\in \Z_+$. Then the following assertions hold:
\begin{enumerate}
\item[(a)] If $|A|=\min\{|\kappa(z)|:z\in S_{p}\}$ and that the range of $\kappa$ intersects $\{w\in\mathbb{C}:|w|=|A|\}$ at finitely many distinct points $A_{1},\ldots,A_{j}$, then $T_{0}$ can be uniquely written as
$$T_{0}=T_{0,1}+T_{0,2}+\cdots+T_{0,j},$$
for some $L^{p}$-tempered distributions $T_{0,i}$ on $\mathcal{X}$ satisfying $\Lambda T_{0,i}=A_{i}T_{0,i}$ for all $i=1,\ldots,j$.
\item[(b)] If $|A|<\min\{|\kappa(z)|:z\in S_{p}\}$, then $T_{0}=0$.
\end{enumerate}
\end{Theorem}

Both of the parts of this theorem can be proved following steps similar to that of previous theorem. The only difference is  instead of (\ref{recurrence1}), here we use the recurrence relation $\kappa(z)^{k}\widehat{T}_{0}=A^{k}\widehat{T}_{k}$ for all $k\in\mathbb{Z}_{+}$, in all intermediate steps. This clarifies why an infinite sequence of $T_k$ is enough for this case.

The next result shows that if we assume  additionally that  the symbol $\kappa(z)$ of the multiplier $\Lambda$ is nonvanishing, then hypothesis of Theorem \ref{maximumcase} can be cut-down to one-sided sequence.

\begin{Theorem}\label{maximumcasezero}
Let $1\leq p<2$. Let $\Lambda$ be a multiplier on $\mathcal{S}_{p}(\mathcal{X})$ with symbol $\kappa(z)$ satisfying $\kappa(z)\neq 0$ for all $z\in S_{p}$. Suppose that $\{T_{-k}\}_{k\in\mathbb{Z}_{+}}$ is an infinite sequence in $\mathcal{S}_{p}(\mathcal X)^{\prime}$ satisfying hypothesis (i) of Theorem \ref{maximumcase} for all $-k$, $k\in \mathbb{N}$ and hypothesis (ii) of Theorem \ref{maximumcase} for all $-k$, $k\in \mathbb{Z}_{+}$. Then both the assertions (a) and (b) of Theorem \ref{maximumcase} hold.
\end{Theorem}

\begin{proof}
Since  $\kappa(z)$ does not vanish for any $z\in S_{p}$, $\Lambda^{-1}$ defined by the symbol $\kappa(z)^{-1}$ is also a multiplier on $\mathcal{S}_{p}(\mathcal{X})$ by Proposition \ref{lp_multiplier}. Setting $T^{\prime}_{k}=T_{-k}$ for all $k\in\mathbb{Z}_{+}$ and $A^{\prime}=1/A$,  we can reduce the hypothesis to that of Theorem \ref{minimumcase} and the assertion follows from there.
\end{proof}

\section{Strichartz's theorem for a restricted subclass of multipliers}\label{refined}
In this section we shall apply all previously proven results (from Section \ref{pcaseroestrichartz}) to a restricted subclass of multipliers on $\mathcal{S}_{p}(\mathcal{X})$, where the sequence of $L^{p}$-tempered distributions will be replaced by functions having appropriate growth. This will allow us to further decompose the eigenfunctions of the multipliers to eigenfunctions of the Laplacian, which can be realized as Poisson transform of functions on the boundary of $\mathcal{X}$. The following propositions are steps towards this goal.

\begin{Proposition}\label{regulardistributions}
Let $1\leq p_{1}\leq p_{2}< 2$. If $f\in L^{p_{2}^\prime,\infty}(\mathcal{X})$  then there exists a semi-norm $\nu_{p_{1},m}(\cdot)$ of $\mathcal{S}_{p_{1}}(\mathcal{X})$ and a constant $C>0$ such that
$$|\langle f,\phi \rangle|\leq C\|f\|_{p_{2}^\prime,\infty}~\nu_{p_{1},m}(\phi),\quad\text{for all }\phi\in\mathcal{S}_{p_{1}}(\mathcal{X}).$$
\end{Proposition}

\begin{proof}
When $p_{1}=p_{2}$, we refer to \cite[Lemma 4.1]{R3} for the proof. Now using this result and noting that $\nu_{p_{2},m}(\phi)\leq \nu_{p_{1},m}(\phi)$ for all $\phi\in\mathcal{S}_{p_{1}}(\mathcal{X})$, the desired result follows whenever $p_{1}<p_{2}$.
\end{proof}

Our next aim is to derive certain size estimates of the derivatives of $z\mapsto\phi_{z}$. Below,  $D^{m} = d^{m}/dz^{m}$ and $D^{0}=I$. Differentiating the explicit expression (\ref{eqsf}) of the elementary spherical function $m$ times and using the Leibniz rule, we get, for $z\in\mathbb{C}\setminus(\tau/2)\mathbb{Z}$,
\begin{equation}\label{derivativeexpression2}
D^{m}\phi_{z}(x)=\sum\limits_{j=0}^{m}\binom{m}{j}(i|x|\log q)^{j}q^{-|x|/2} \left(q^{iz|x|}D^{(m-j)}\mathbf{c}(z)+(-1)^{j}q^{-iz|x|}D^{(m-j)}\mathbf{c}(-z)\right).
\end{equation}

We need the following lemma for the next proposition (see \cite{KR} for details).
\begin{Lemma}[{{\cite[Lemma 2.5]{KR}}}]\label{lemmaweak}
If $1<p<2$ then for every $f\in L^{p^{\prime},\infty}(\mathcal{X})$, there exists a positive constant $C$ independent of $f$ such that
$$\frac{1}{n}\sum\limits_{x\in B(o,n)}|f(x)|^{p^{\prime}}\leq C\|f\|^{p^{\prime}}_{L^{p^{\prime},\infty}},\quad\text{for all }n\in\mathbb{N}.$$
\end{Lemma}

\begin{Proposition}\label{derivativeestimates}
Let $1\leq p<2$. Suppose that $z_{0}=\alpha+i\delta_{p^{\prime}}$, $\alpha\in\mathbb{R}$. If for a polynomial $P$ with complex coefficients, $P(D)\phi_{z}|_{z=z_{0}}\in L^{p^{\prime},\infty}(\mathcal{X})^{\#}$ then $P$ is a constant polynomial.
\end{Proposition}

\begin{proof}
Let $P(D)=c_0I+c_1D+\cdots+c_ND^N$,  where $N\in\mathbb{N}$ and $c_{1},\cdots,c_{N}\in \C$ are not all zero. Let $k$ be the  largest among $\{1, \ldots, N\}$
such that $c_{k}\neq 0$. We consider the function
$$f(x)=\sum\limits_{m=0}^{k}c_{m}D^{m}\phi_{z}(x)|_{z=z_{0}},\quad\text{for all }x\in\mathcal{X}.$$
Substituting $z_{0}=\alpha+i\delta_{p^{\prime}}$ in (\ref{derivativeexpression2}), it follows that $f(x)=(i|x|\log q)^{k}q^{(i\alpha-1/p^{\prime})|x|}g(x)$, where
\begin{multline*}
g(x)=\sum\limits_{m=0}^{k}c_{m}\left[\sum\limits_{j=0}^{m}\binom{m}{j}(i|x|\log q)^{j-k}\left.\left(D^{(m-j)}\mathbf{c}(z_{0})\right.\right.\right.\\
\left.\left.+(-1)^{j}q^{-2i\alpha|x|}q^{(1-2/p)|x|}D^{(m-j)}\mathbf{c}(-z_{0})\right)\right].
\end{multline*}
Clearly, $|g(x)|\rightarrow|\mathbf{c}(z_{0})||c_{k}|$ as $|x|\rightarrow\infty$. Furthermore, our assumption on $c_{k}$ and the choice of $z_{0}$ (see (\ref{harishchandra})) implies that the limit is non-zero. Therefore there exist a constant $C>0$ and $M\in\mathbb{N}$ such that
\begin{equation}\label{derivativepointwise}
|f(x)|\geq C|\mathbf{c}(z_{0})||c_{k}| |x|^{k}q^{-|x|/p^{\prime}}, \quad\text{for all }x\in\mathcal{X}\text{ such that }|x|\geq M.
\end{equation}
When $p=1$, the expression above  contradicts the fact that $f\in L^{\infty}(\mathcal{X})^{\#}$. For $1<p<2$, from (\ref{derivativepointwise}) and Lemma \ref{lemmaweak}, we get $n^{kp^{\prime}-1}\leq C\|f\|^{p^{\prime}}_{L^{p^{\prime},\infty}(\mathcal{X})}$ for all $n\geq M$. This contradicts that $f\in L^{p^{\prime},\infty}(\mathcal{X})^{\#}$. Hence $c_{1},\cdots,c_{N}$ are all zero and $P$ is a constant polynomial.
\end{proof}

To proceed further, we need a boundary representation of the generalized eigenfunctions of the Laplacian $\mathcal{L}$ on $\mathcal{X}$. If a function $f$ on $\mathcal{X}$ satisfies $\mathcal{L}^{N}f=0$, Cohen et al. \cite[Theorem 4.1]{CCGS} proved that $f$ can be written as sum of the derivatives of $\mathcal{P}_{i\delta_{\infty}}\lambda$, where $\lambda$ is a finitely additive measure on $\Omega$, that is, a complex-valued set function on the Borel $\sigma$-algebra generated by the subsets $E(x)$ (see (\ref{sectors})) of $\Omega$ satisfying finite additivity, and
\begin{equation}\label{poissonmeasure}
\mathcal{P}_{z}\lambda(x)=\int\limits_{\Omega}p^{1/2+iz}(x,\omega)d\lambda(\omega),\quad\text{for all }x\in\mathcal{X}.
\end{equation}
Thereafter, Picardello et al. \cite{PW} extended this result for the generalized eigenfunctions of $\mathcal{L}$ corresponding to arbitrary complex numbers $\gamma(z)$. Before stating the result precisely, we would like to point out certain important differences between our notation and those in \cite{PW}: the Laplacian $\mathcal{L}$ defined by (\ref{laplaciandefinition}) corresponds to the operator $(I-P)$ in \cite[p. 557]{PW}, the function $\gamma(z)$ defined by (\ref{gammaz}) corresponds to $1-\lambda$ in \cite[p. 557]{PW} and the height function $h_{\omega}(x)$ (in Subsection \ref{laplacianpoissontransform}) corresponds to $-\mathfrak{h}(x,\xi)$ in \cite[p. 557]{PW}. Further, the function $K(x,\xi|\lambda)$ from \cite[p. 557]{PW} corresponds to $p^{1/2+iz}(x,\omega)$ in (\ref{poissonmeasure}), where $1-\lambda=\gamma(z)$ and $\Im z<0$. In view of this,  the integrand in \cite[Corollary 5.4]{PW} can be written in our notation as:
$$K(x,\xi|\lambda)~\mathfrak{h}(x,\xi)^{k}=\frac{(-1)^{k}}{(i\log q)^{k}}~D^{k}p^{1/2+iz}(x,\omega),\quad\text{for all }k\in\mathbb{Z}_{+}.$$
Finally, applying the dominated convergence theorem to the integrals in \cite[Corollary 5.4]{PW} and using the expression above, we get the following variant:

\begin{Proposition}[{{\cite[Theorem 5.3 and Corollary 5.4]{PW}}}]\label{generalizedeigenfunction}
Let $z_{0}=\alpha+i\delta_{p^{\prime}}$ for some $\alpha\in\mathbb{R}$ and for some $1\leq p<2$. A complex-valued function $f$ on $\mathcal{X}$ satisfies $(\mathcal{L}-\gamma(z_{0})I)^{N}f=0$ for some $N\in\mathbb{N}$ if and only if there exist finitely additive measures $\lambda_{0},\ldots,\lambda_{N-1}$ on $\Omega$ such that
\begin{equation}\label{generalizedeigenfunctions2}
f(x)=\sum\limits_{m=0}^{N-1}D^{m}(\mathcal{P}_{z}\lambda_{m})(x)|_{z=z_{0}},\quad\text{for all }x\in\mathcal{X}.
\end{equation}
\end{Proposition}

An immediate corollary of Proposition \ref{generalizedeigenfunction} is the following:
\begin{Corollary}\label{radialgeneralizedeigenfunction}
Let $z_{0}=\alpha+i\delta_{p^{\prime}}$ for some $\alpha\in\mathbb{R}$ and for some $1\leq p<2$. A radial function $f$ on $\mathcal{X}$ satisfies $(\mathcal{L}-\gamma(z_{0})I)^{N}f=0$ for some $N\in\mathbb{N}$ if and only if there exist complex constants $c_{0},c_{1},\ldots,c_{N-1}$ such that
\begin{equation}\label{radialgeneralizedeigenfunctions2}
f(x)=\sum\limits_{m=0}^{N-1}c_{m}D^{m}\phi_{z}(x)|_{z=z_{0}},\quad\text{for all }x\in\mathcal{X}.
\end{equation}
\end{Corollary}

\begin{proof}
If a complex-valued function $f$ is of the form (\ref{radialgeneralizedeigenfunctions2}) then using Proposition \ref{generalizedeigenfunction} with $\lambda_{m}=c_{m}\nu$, where $c_{m}\in\mathbb{C}$ and $m=0,\ldots,N-1$, we get $(\mathcal{L}-\gamma(z_{0})I)^{N}f=0$. The fact that $f$ is radial is an easy consequence of the explicit expression (\ref{derivativeexpression2}).

Conversely, if we assume that $f$ satisfies $(\mathcal{L}-\gamma(z_{0})I)^{N}f=0$ for some $N\in\mathbb{N}$ then according to Proposition \ref{generalizedeigenfunction}, there exist finitely additive measures $\lambda_{0},\ldots,\lambda_{N-1}$ on $\Omega$ such that (\ref{generalizedeigenfunctions2}) holds. Now using the hypothesis that $f$ is radial and hence $f(x)=\mathscr{R}f(x)$ in (\ref{generalizedeigenfunctions2}), we get
\begin{equation}\label{generalizedeigenfunctions3}
f(x)=\sum\limits_{m=0}^{N-1}\mathscr{R}(D^{m}(\mathcal{P}_{z}\lambda_{m}))(x)|_{z=z_{0}},\quad\text{for all }x\in\mathcal{X}.
\end{equation}

Since $\mathcal{L}$ commutes with the radialization operator $\mathscr{R}$ (see (\ref{radializationoperator})) and $\mathcal{L}(\mathcal{P}_{z_{0}}\lambda)= \gamma(z_{0})\mathcal{P}_{z_{0}}\lambda$ for every finitely additive measure $\lambda$ on $\Omega$, it follows that $\mathscr{R}(\mathcal{P}_{z_{0}}\lambda))$ is a radial eigenfunction of $\mathcal{L}$ with eigenvalue $\gamma(z_{0})$. Consequently, $\mathscr{R}(\mathcal{P}_{z_{0}}\lambda)=c\phi_{z_{0}}$. Putting $x=o$ in the previous expression and using the fact that $\phi_{z_{0}}(o)=1$, we have
\begin{equation}\label{generalizedeigenfunctions4}
\mathscr{R}(\mathcal{P}_{z_{0}}\lambda)(x)=\lambda(\Omega)~\phi_{z_{0}}(x),\quad\text{for all }x\in\mathcal{X}.
\end{equation}
Furthermore, by a standard use of the dominated convergence theorem we obtain $\mathscr{R}(D^{m}(\mathcal{P}_{z}\lambda))|_{z=z_{0}}=D^{m}(\mathscr{R}(\mathcal{P}_{z}\lambda))|_{z=z_{0}}$ for all $m\in\mathbb{Z}_{+}$. Combining this with (\ref{generalizedeigenfunctions4}), we get
\begin{equation}\label{generalizedeigenfunctions5}
\mathscr{R}(D^{m}(\mathcal{P}_{z}\lambda))(x)|_{z=z_{0}}=\lambda(\Omega)~D^{m}\phi_{z}(x)|_{z=z_{0}},\quad\text{for all }x\in\mathcal{X},\text{ and for all }m\in\mathbb{Z}_{+}.
\end{equation}
Finally, the desired result follows by using (\ref{generalizedeigenfunctions5}) on the right hand side of (\ref{generalizedeigenfunctions3}).
\end{proof}

Let $\Psi$ be a nonconstant holomorphic  function defined on a connected open set $O\subseteq\mathbb{C}$ containing $\gamma(S_{p})$. Then it can be verified  that $\Psi\circ\gamma$ is in $\mathcal{H}(S_{p})^{\#}$ for all $1\leq p< 2$. Hence by Proposition \ref{lp_multiplier}, $\Psi\circ\gamma$ corresponds to a multiplier on $\mathcal{S}_{p}(\mathcal{X})$, which will be denoted by $\Psi(\mathcal{L})$. With this preparation, we now offer the following main results of this section. 

\begin{Theorem}\label{pcaseforfunctionsoflaplacianmaximumzero}
For $1\leq p<2$, let $\Psi(\mathcal{L})$ be a multiplier on $\mathcal{S}_{p}(\mathcal{X})$ associated with the symbol $\Psi\circ\gamma$. Suppose that $\{f_{k}\}_{k\in\mathbb{Z}}$ is a bi-infinite sequence of functions on $\mathcal{X}$ satisfying for all $k\in\mathbb{Z}$:
\begin{enumerate}
\item[(i)] $\Psi(\mathcal{L}) f_{k}=A~f_{k+1}$ for some $A\in\mathbb{C}^{\times}$.
\item[(ii)] $\|f_{k}\|_{L^{p^{\prime},\infty}(\mathcal{X})}\leq M$ for some constant $M>0$.
\end{enumerate}
If $|A|=\max\{|\Psi\circ\gamma(z)|:z\in S_{p}\}$ and if the range of $\Psi\circ\gamma$ intersects $\{w\in\mathbb{C}:|w|=|A|\}$ at finitely many distinct points, then $f_{0}$ can be uniquely written as $f_{0}=g_{1}+\cdots+g_{m}$ for some $g_{l}\in L^{p^{\prime},\infty}(\mathcal{X})$, $l=1,\ldots,m$, satisfying $\mathcal{L}g_{l}=\gamma(\alpha_{l}+i\delta_{p^{\prime}})g_{l}$, where $-\tau/2< \alpha_{l}\leq\tau/2$ are distinct and $|\Psi\circ\gamma(\alpha_{l}+i\delta_{p^{\prime}})|=|A|$. In particular, for all $l=1,\ldots,m$, $g_{l}=\mathcal{P}_{\alpha_{l}+i\delta_{p^{\prime}}}F_{l}$ for some $F_{l}\in L^{p^{\prime}}(\Omega)$.
\end{Theorem}

\begin{proof}
As done in the proof of Theorem \ref{maximumcase}, we divide the proof into the following steps:

\noindent{\textbf{Step 1:}} We first assume that the sequence $\{f_{k}\}_{k\in\mathbb{Z}}$ consists of only radial functions. Then the hypotheses of this theorem together with Proposition \ref{regulardistributions} (with $p_{1}=p_{2}=p$) implies that $\{f_{k}\}_{k\in\mathbb{Z}}$ is a bi-infinite sequence of radial $L^{p}$-tempered distributions on $\mathcal{X}$ which satisfies  the hypotheses of Theorem \ref{maximumcase} with $\Lambda=\Psi(\mathcal{L})$. Moreover, $\Psi(\mathcal{L})$ is nontrivial as $\Psi$ is a nonconstant. Therefore (\ref{radiallpdistribution}) holds with $f_{0}$ in place of $T_{0}$, that is,
$$(\Psi(\mathcal{L})-A_{1}I)(\Psi(\mathcal{L})-A_{2}I)\cdots(\Psi(\mathcal{L})-A_{j}I)f_{0}=0,$$
where $\{\Psi\circ\gamma(z):z\in S_{p}\}\cap\{w\in\mathbb{C}:|w|=|A|\}=\{A_{1},A_{2},\ldots,A_{j}\}$. By Remark \ref{remark-on-weak-lp}, we have a unique decomposition $f_{0}=f_{0,1}+f_{0,2}+\cdots+f_{0,j}$, for some functions $f_{0,i}\in L^{p^{\prime},\infty}(\mathcal{X})^{\#}$ satisfying
\begin{equation}\label{eigenfunctionmultiplier}
(\Psi(\mathcal{L})-A_{i}I) f_{0,i}=0,\quad\text{for all }i=1,\ldots,j.
\end{equation}

Fix $i\in\{1,\ldots,j\}$. By the maximum modulus principle it follows that the zeros of $\Psi-A_{i}$ with $\gamma(S_{p})$ as its domain, lie on the boundary $\gamma(\partial S_{p})$. On the other hand, since $\Psi-A_{i}$ is analytic on a connected open set $O$ containing $\gamma(S_{p})$ and $\gamma(\partial S_{p})$ is compact, using the identity theorem we infer that $\Psi-A_{i}$ has finitely many zeros in $\gamma(S_{p})$. Let $\{w_{1},\ldots,w_{m_{i}}\}\subset \gamma(\partial S_{p})$ be the zeros of $\Psi-A_{i}$ with multiplicities $N_{1},\ldots,N_{m_{i}}$ respectively. Using the analyticity of $\Psi$, for all $w\in O$, we obtain
$$\Psi(w)-A_{i}=(w-w_{1})^{N_{1}}\cdots(w-w_{m_{i}})^{N_{m_{i}}}\Phi(w),$$
for some analytic function $\Phi$ on $O$ satisfying $\Phi(w)\neq 0$ for all $w\in\gamma(S_{p})$. Combining this fact with (\ref{eigenfunctionmultiplier}), we get
$$\langle (\gamma(z)-w_{1})^{N_{1}}\cdots(\gamma(z)-w_{m_{i}})^{N_{m_{i}}}\widehat{f}_{0,i},\psi\cdot \Phi\circ\gamma\rangle=0,\quad\text{for all }\psi\in\mathcal{H}(S_{p})^{\#}.$$
Since $\Phi(w)\neq 0$ for all $w\in\gamma(S_{p})$, hence so is $\Phi\circ\gamma(z)$ for all $z\in S_{p}$. Therefore $(\Phi\circ\gamma)^{-1}\in\mathcal{H}(S_{p})^{\#}$, which together with the expression above and  Theorem \ref{lp isom} implies that
$$(\gamma(z)-w_{1})^{N_{1}}\cdots(\gamma(z)-w_{m_{i}})^{N_{m_{i}}}\widehat{f}_{0,i}=0.$$
Taking inverse Fourier transform using Theorem \ref{lp isom} we finally have
\begin{equation}\label{generalizedeigenfunctionlaplacian}
(\mathcal{L}-w_{1}I)^{N_{i}}\cdots(\mathcal{L}-w_{m_{i}}I)^{N_{i}}f_{0,i}=0,\quad\text{where }N_{i}=\max\{N_{1},\ldots,N_{m_{i}}\}.
\end{equation}

\noindent{\textbf{Step 2:}} In this step, we shall prove that $N_{i}=1$ in (\ref{generalizedeigenfunctionlaplacian}). Seeking a contradiction, we suppose that $(\mathcal{L}-w_{1}I)\cdots(\mathcal{L}-w_{m_{i}}I)f_{0,i}\neq 0$ and that (\ref{generalizedeigenfunctionlaplacian}) holds for some $N_{i}\geq 2$. Then there exists $1\leq N_{0,i}\leq N_{i}-1$ such that
\begin{align}
(\mathcal{L}-w_{1}I)^{N_{0,i}}\cdots(\mathcal{L}-w_{m_{i}}I)^{N_{0,i}}f_{0,i}&\neq 0,\quad\text{and}\label{contradictionfunction}\\
(\mathcal{L}-w_{1}I)^{N_{0,i}+1}\cdots(\mathcal{L}-w_{m_{i}}I)^{N_{0,i}+1}f_{0,i}&=0.\label{generalisedpropositionfunction}
\end{align}
Then Remark \ref{remark-on-weak-lp} combined with (\ref{generalisedpropositionfunction}) and the fact that $f_{0,i}\in L^{p^{\prime},\infty}(\mathcal{X})^{\#}$ (see Step 1) gives us the decomposition $f_{0,i}=g_{1}+\cdots+g_{m_{i}}$, for some functions $g_{k}\in L^{p^{\prime},\infty}(\mathcal{X})^{\#}$ satisfying
\begin{equation}\label{contradictionfunction3}
(\mathcal{L}-w_{k}I)^{N_{0,i}+1}g_{k}=0,\quad\text{for all }k=1,\ldots,m_{i}.
\end{equation}
Furthermore, from (\ref{contradictionfunction}), there exists $k_{0}\in\{1,\ldots,m_{i}\}$ such that
\begin{equation}\label{contradictionfunction2}
(\mathcal{L}-w_{k_{0}}I)^{N_{0,i}}g_{k_{0}}\neq 0.
\end{equation}
On the other hand, since $w_{k_{0}}=\gamma(\alpha_{k_{0}}+i\delta_{p^{\prime}})$ for a unique $\alpha_{k_{0}}\in(-\tau/2,\tau/2]$, using (\ref{contradictionfunction3}) and Corollary \ref{radialgeneralizedeigenfunction}, we obtain the following boundary representation of $g_{k_{0}}$:
$$g_{k_{0}}(x)=\sum\limits_{r=0}^{N_{0,i}}c_{r}D^{r}\phi_{z}(x)|_{z=\alpha_{k_{0}}+i\delta_{p^{\prime}}},\quad\text{where }c_{0},\ldots,c_{N_{0,i}}\text{ are complex constants}.$$
Now Proposition \ref{derivativeestimates} together with the fact $g_{k_{0}}\in L^{p^{\prime},\infty}(\mathcal{X})^{\#}$ implies that $c_{r}=0$ for all $r=1,\ldots,c_{N_{0,i}}$. Consequently, $(\mathcal{L}-w_{k_{0}}I)g_{k_{0}}=0$ which contradicts (\ref{contradictionfunction2}). Hence $N_{i}=1$ in (\ref{generalizedeigenfunctionlaplacian}). Now using this fact together with Remark \ref{remark-on-weak-lp} we get a unique decomposition $f_{0,i}=g_{1}+\cdots+g_{m_{i}}$, for some functions $g_{k}\in L^{p^{\prime},\infty}(\mathcal{X})^{\#}$ satisfying $(\mathcal{L}-w_{k}I)g_{k}=0$, where $w_{k}=\gamma(\alpha_{k}+i\delta_{p^{\prime}})$ and $\alpha_{k}\in(-\tau/2,\tau/2]$ are all distinct, $k=1\ldots,m_{i}$. Finally, $g_{k}=c_{k}\phi_{\alpha_{k}+i\delta_{p^{\prime}}}$ follows trivially (see Subsection \ref{laplacianpoissontransform}).

\noindent\textbf{Step 3:} The rest of the argument is similar to Step 3 of the proof of Theorem \ref{maximumcase}. For the sake of completeness, we provide a quick sketch.  Let us assume that $\{f_{k}\}_{k\in\mathbb{Z}}$ does not consist of only radial functions. Then for every $g$ in $G$, $\{\mathscr{R}(\tau_{g}f_{k})\}_{k\in\mathbb{Z}}$ is a bi-infinite sequence of radial functions on $\mathcal{X}$  which satisfy both the hypothesis of this theorem. Hence by using the result for radial functions, we obtain $(\mathcal{L}-\gamma(\alpha_{1}+i\delta_{p^{\prime}}))\cdots(\mathcal{L}-\gamma(\alpha_{m}+i\delta_{p^{\prime}}))\mathscr{R}(\tau_{g}f_{0})=0$. Consequently, $\mathscr{R}(\tau_{g}((\mathcal{L}-\gamma(\alpha_{1}+i\delta_{p^{\prime}}))\cdots(\mathcal{L}-\gamma(\alpha_{m}+i\delta_{p^{\prime}}))f_{0}))=0$ for all $g$ in $G$ and hence so is $(\mathcal{L}-\gamma(\alpha_{1}+i\delta_{p^{\prime}}))\cdots(\mathcal{L}-\gamma(\alpha_{m}+i\delta_{p^{\prime}}))f_{0}$. Applying Proposition \ref{generalizedproposition} and Theorem \ref{weaklpchar}, we get the desired conclusion.
\end{proof}

 We also have the following variations of Theorem \ref{pcaseforfunctionsoflaplacianmaximumzero}, which will be used.
We recall that $\Z_+$ is the set of nonnegative integers.

\begin{Theorem}\label{pcaseforfunctionsoflaplacianmaximumnonzero}
For $1\leq p<2$, let $\Psi(\mathcal{L})$ be a multiplier on $\mathcal{S}_{p}(\mathcal{X})$ with symbol $\Psi\circ\gamma$ such that $\Psi\circ\gamma(z)\neq 0$ for all $z\in S_{p}$. Let $\{f_{k}\}_{k\in\mathbb{Z}}$ be a bi-infinite sequence of functions on $\mathcal{X}$ satisfying hypothesis (ii) of Theorem \ref{pcaseforfunctionsoflaplacianmaximumzero}.
\begin{enumerate}
\item[(a)] Suppose that the sequence  $\{f_{k}\}_{k\in\mathbb{Z}_{+}}$ satisfies condition (i) of Theorem \ref{pcaseforfunctionsoflaplacianmaximumzero} for all $k\in\mathbb{Z}_{+}$. If  $|A|= \min\{|\Psi\circ\gamma(z)|:z\in S_{p}\}$ and if the range of $\Psi\circ\gamma$ intersects $\{w\in\mathbb{C}:|w|=|A|\}$ at finitely many distinct points, then the same conclusion holds for $f_0$.

\item[(b)] Suppose that the sequence  $\{f_{-k}\}_{k\in\mathbb{Z}_{+}}$ satisfies condition (i) of Theorem \ref{pcaseforfunctionsoflaplacianmaximumzero} for all $-k$, $k\in\mathbb{N}$. If $|A|= \max\{|\Psi\circ\gamma(z)|:z\in S_{p}\}$ and if the range of $\Psi\circ\gamma$ intersects $\{w\in\mathbb{C}:|w|=|A|\}$ at finitely many distinct points, then the same conclusion holds for $f_0$.
\end{enumerate}
\end{Theorem}

\begin{proof}
The proof of (a) is similar to that of Theorem \ref{pcaseforfunctionsoflaplacianmaximumzero} with the only difference that instead of using Theorem \ref{maximumcase}, here we combine Remark \ref{remark-on-weak-lp} and Theorem \ref{minimumcase} to get a decomposition of the form $f_{0}=f_{0,1}+f_{0,2}+\cdots+f_{0,j}$, for some weak $L^{p}$-functions $f_{0,i}$ on $\mathcal{X}$ satisfying (\ref{eigenfunctionmultiplier}).

For (b), once again we use Remark \ref{remark-on-weak-lp} and Theorem \ref{maximumcasezero} to get a unique decomposition of $f_{0}$ as the sum of eigenfunctions of $\Psi(\mathcal{L})$. Finally, the desired result follows by imitating the proof of Theorem \ref{pcaseforfunctionsoflaplacianmaximumzero}.
\end{proof}


We shall further narrow down our focus to some particular examples of multipliers, namely the Laplacian itself, sphere and ball averaging operators and complex time heat operators. In a similar context these multipliers were considered in various other settings (see for example \cite{BKS,KCC,NS,S} and the references therein). 

\subsection{The Laplacian as a multiplier}\label{laplacian}
We begin by considering the Laplacian $\mathcal{L}$ on $\mathcal{X}$ as a multiplier on $\mathcal{S}_{p}(\mathcal{X})$ with symbol $\gamma(z)$ (see (\ref{gammaz})). The result we obtain here will compliment those in \cite{R3}.

Substituting $z=\alpha+i\delta_{p^{\prime}}$ in (\ref{gammaz}) we get  real and imaginary parts of the symbol $\gamma(z)$ of $\mathcal{L}$ as:
\begin{equation}\label{realimaginarygamma}
\Re\gamma(z)=1-\frac{q^{1/p^{\prime}}+q^{1/p}}{q+1}\cos(\alpha\log q)\quad\text{and}\quad\Im\gamma(z)=\frac{q^{1/p^{\prime}}-q^{1/p}}{q+1}\sin(\alpha\log q).
\end{equation}
Thus $\gamma$ maps the lines $\{z\in\mathbb{C}:z=\alpha\pm i\delta_{p},\alpha\in\mathbb{R}\}$ onto concentric ellipses in the complex plane, which degenerates into a line segment when $p=2$. See \cite[Section 3]{R3} for details. More precisely, for $1\leq p\leq r\leq 2$,
$$\gamma(S_{2})\subseteq\gamma(S_{r})\subseteq\gamma(S_{p})\subseteq\gamma(S_{1}).$$
We have following consequences of the theorems proved above.

\begin{Corollary}\label{maxzerolaplacian}
Suppose that $\{f_{k}\}_{k\in\mathbb{Z}}$ is a bi-infinite sequence of functions on $\mathcal{X}$ such that $\|f_{k}\|_{L^{\infty}(\mathcal{X})}\leq M$ for all $k\in\mathbb{Z}$ and for some $M>0$. If $\mathcal{L}f_{k}=A~f_{k+1}$ for all $k\in\mathbb{Z}$, where $A\in\mathbb{C}$ satisfies $|A|=\gamma(\tau/2+i\delta_{\infty})$, then $\mathcal{L}f_{0}=\gamma(\tau/2+i\delta_{\infty})f_{0}$. In particular, $f_{0}=\mathcal{P}_{\tau/2+i\delta_{\infty}}F$ for a unique $F\in L^{\infty}(\Omega)$.
\end{Corollary}

\begin{proof}
The proof of this corollary is a direct consequence of Theorem \ref{pcaseforfunctionsoflaplacianmaximumzero} with $p=1$ and $\Psi(w)=w$. Indeed, a simple computation using (\ref{realimaginarygamma}) shows that $\max\{|\gamma(z)|:z\in S_{1}\}=\gamma(\tau/2+i\delta_{\infty})$ and $|\gamma(\alpha+i\delta_{\infty})|<\gamma(\tau/2+i\delta_{\infty})$ for all $-\tau/2< \alpha<\tau/2$.
\end{proof}

\begin{Corollary}
Let $1<p<2$. Suppose that $\{f_{k}\}_{k\in\mathbb{Z}}$ is a bi-infinite sequence of functions on $\mathcal{X}$ such that $\|f_{k}\|_{L^{p^{\prime},\infty}(\mathcal{X})}\leq M$ for all $k\in\mathbb{Z}$ and for some $M>0$.
\begin{enumerate}
\item[(a)] If $z_{1}=i\delta_{p^{\prime}}$ and if $\mathcal{L}f_{k}=A~f_{k+1}$ for all $k\in\mathbb{Z}_{+}$, where $A\in\mathbb{C}$ satisfies $|A|=\gamma(z_{1})$, then $\mathcal{L}f_{0}=\gamma(z_{1})f_{0}$.
\item[(b)] If $z_{2}=\tau/2+i\delta_{p^{\prime}}$ and if $\mathcal{L}f_{-k}=A~f_{-k+1}$ for all $k\in\mathbb{N}$, where $A\in\mathbb{C}$ satisfies $|A|=\gamma(z_{2})$, then $\mathcal{L}f_{0}=\gamma(z_{2})f_{0}$.
\end{enumerate}
In both cases, $f_{0}=\mathcal{P}_{z_{j}}F_{j}$ for a unique $F_{j}\in L^{p^{\prime}}(\Omega)$, where $j=1,2$.
\end{Corollary}

\begin{proof}
A crucial point to note here is that when $1<p<2$, $\gamma(z)\neq 0$ for all $z\in S_{p}$. Therefore unlike the case $p=1$ (see Corollary \ref{maxzerolaplacian}), here we use Theorem \ref{pcaseforfunctionsoflaplacianmaximumnonzero} with $\Psi(w)=w$. Using (\ref{realimaginarygamma}), it follows that $\min\{|\gamma(z)|:z\in S_{p}\}=\gamma(i\delta_{p^{\prime}})$ and $\max\{|\gamma(z)|:z\in S_{p}\}=\gamma(\tau/2+i\delta_{p^{\prime}})$. Furthermore,
$$\gamma(i\delta_{p^{\prime}})<|\gamma(\alpha+i\delta_{p^{\prime}})|<\gamma(\tau/2+i\delta_{p^{\prime}}),\quad\text{for all }-\tau/2< \alpha<\tau/2,$$
which gives us the desired conclusion.
\end{proof}


\subsection{Averaging operators}
Let  $\chi_{S(o,n)}$ and $\chi_{B(o,n)}$  respectively be the indicator functions of the sphere $S(o,n)$ and the ball $B(o, n)$. The average  of a  function $f$ over $S(x, n)$ is given by
$$\mathscr{S}_{n}f(x)=\frac{1}{\# S(o,n)}f\ast \chi_{S(o,n)}(x)=\frac{1}{\# S(o,n)}\sum\limits_{y\in S(x,n)}f(y).$$
Clearly, $\mathscr{S}_{0}f=f$ and $\mathscr{S}_{1}f=f-\mathcal{L}f$ and for  $n\geq 2$ (see \cite[Lemma 1]{FTP1}),
\begin{equation}\label{sphericalaverage}
\mathscr{S}_{n}f=\frac{q+1}{q}\mathscr{S}_{n-1}(\mathscr{S}_{1}f)-\frac{1}{q}\mathscr{S}_{n-2}f.
\end{equation}
Consequently, for each $n$, $\mathscr{S}_{n}=p_{n}(\mathcal{L})$, where $p_{n}$ is a polynomial of degree $n$. On the other hand, the  average of $f$ over the geodesic ball $B(x,n)$ is defined as
$$\mathscr{B}_{n}f(x)= \frac{1}{\# B(o,n)} f\ast \chi_{B(o, n)}(x)=\frac{1}{\# B(o,n)}\sum\limits_{j=0}^{n} (\# S(o,j))\mathscr{S}_{j}f(x),\quad\text{for all }n\in\mathbb{Z}_{+}.$$
From  (\ref{sphericalaverage}) we infer that $\mathscr{B}_{n}$ is also an $n$-degree polynomial of $\mathcal{L}$. Hence both $\mathscr{S}_{n}$ and $\mathscr{B}_{n}$ are multipliers on $\mathcal{S}_{p}(\mathcal{X})$, $1\leq p< 2$, with symbols $\phi_{z}(n)$ and $\psi_{z}(n)$ respectively, where $\phi_{z}(n)$ denotes the value of $\phi_{z}$ (see (\ref{eqsf})) on $S(o,n)$ and
\begin{equation}\label{ballsymbol}
\psi_{z}(n)=\frac{1}{\# B(o,n)}\sum\limits_{j=0}^{n}(\# S(o,j))\phi_{z}(j),\quad\text{for all }n\in\mathbb{Z}_{+}.
\end{equation}
We now consider the analogues of Strichartz's theorem for these averaging operators.

\begin{Corollary}\label{pcaseaveragingoperators}
Fix $n\in\mathbb{N}$. For $1\leq p<2$, let $\{f_{k}\}_{k\in\mathbb{Z}}$ be a bi-infinite sequence of functions on $\mathcal{X}$ such that for all $k\in\mathbb{Z}$, $\Lambda f_{k}=Af_{k+1}$ for some constant $A\in\mathbb{C}^{\times}$ and $\|f_{k}\|_{L^{p^{\prime},\infty}(\mathcal{X})}\leq M$ for some $M>0$.
\begin{enumerate}
\item[(i)] If $\Lambda=\mathscr{S}_{n}$ and if $|A|=\phi_{i\delta_{p^{\prime}}}(n)$ then $f_{0}$ can be uniquely written as $f_{0}=g_{1}+g_{2}$ for some $g_1, g_2\in L^{p^{\prime},\infty}(\mathcal{X})$ satisfying $\mathcal{L}g_{1}=\gamma(i\delta_{p^{\prime}})g_{1}$ and $\mathcal{L}g_{2}=\gamma(\tau/2+i\delta_{p^{\prime}})g_{2}$.
\item[(ii)] If $\Lambda =\mathscr{B}_{n}$ and if $|A|=\psi_{i\delta_{p^{\prime}}}(n)$ then $\mathcal{L}f_{0}=\gamma(i\delta_{p^{\prime}})f_{0}$.
\end{enumerate}
In particular, $g_{j}=\mathcal{P}_{z_{j}}G_{j}$ and $f_{0}=\mathcal{P}_{z_{1}}F$  for some unique functions $G_{j}, F\in L^{p^{\prime}}(\Omega)$, where $z_{1}=i\delta_{p^{\prime}}$ and $z_{2}=\tau/2+i\delta_{p^{\prime}}$.
\end{Corollary}

\begin{proof}
\rm{(i):} Fix $n\in\mathbb{N}$ and let $1\leq p<2$. Since $z\mapsto\phi_{z}(n)$ is in $\mathcal{H}(S_{p})^{\#}$, appealing to the maximum modulus principle and Theorem \ref{pcaseforfunctionsoflaplacianmaximumzero}, it is enough to prove the following assertions:
\begin{enumerate}
\item[(a)] $\max\{|\phi_{\alpha+i\delta_{p^{\prime}}}(n)|:-\tau/2<\alpha\leq\tau/2\}=\phi_{i\delta_{p^{\prime}}}(n)$,
\item[(b)] $\{\phi_{\alpha+i\delta_{p^{\prime}}}(n):-\tau/2<\alpha\leq\tau/2~\text{and}~|\phi_{z}(n)|=|A|\}=\{\phi_{i\delta_{p^{\prime}}}(n),\phi_{\tau/2+i\delta_{p^{\prime}}}(n)\},$
\item[(c)] The maximum modulus is attained at only two points on the boundary $\{\alpha+i\delta_{p^{\prime}}:-\tau/2<\alpha\leq \tau/2\}$, namely $z_{1}=i\delta_{p^{\prime}}$ and $z_{2}=\tau/2+i\delta_{p^{\prime}}$.
\end{enumerate}

Depending on whether $n\in \N$ is odd or even, we shall subdivide the proof into two cases. Substituting the formula for the c-function (\ref{harishchandra}) into (\ref{eqsf}), we first obtain, for all $z\in\mathbb{C}\setminus(\tau/2)\mathbb{Z}$,
\begin{align} \label{generalphiz}
\phi_{z}(n)&=\frac{1}{q+1}\left[q^{1-n/2}~\frac{q^{iz(n+1)}-q^{-iz(n+1)}}{q^{iz}-q^{-iz}}
-q^{-n/2}~\frac{q^{iz(n-1)}-q^{-iz(n-1)}}{q^{iz}-q^{-iz}}\right].
\end{align}

\noindent\textbf{Case 1:} We assume that $n\in\mathbb{N}$ is odd. From (\ref{generalphiz}), it follows that
\begin{equation}\label{phizodd}
\phi_{z}(2m+1)=\begin{cases}
\frac{q^{1/2}}{q+1} \gamma_{1}(z), & \text{ if }m=0,\\
 & \\
\frac{q^{-m-1/2}}{q+1}[~q~\gamma_{2m+1}(z)+(q-1)\sum\limits_{j=0}^{m-1}\gamma_{2j+1}(z)~], & \text{ if }m\in\mathbb{N},
\end{cases}
\end{equation}
where, for all $j\in\mathbb{Z}_{+}$, we define
\begin{equation}\label{gammajexplicit}
\gamma_{j}(z)=q^{izj}+q^{-izj}.
\end{equation}
We now fix $j$. Substituting $z=\alpha+i\delta_{p^{\prime}}$ into the expression of $\gamma_{j}(z)$ we obtain
\begin{equation}\label{gammaj}
\Re\gamma_{j}(z)=\left(q^{-j\delta_{p^{\prime}}}+q^{j\delta_{p^{\prime}}}\right)\cos(j\alpha\log q)\quad\text{and}\quad\Im\gamma_{j}(z)=\left(q^{-j\delta_{p^{\prime}}}-q^{j\delta_{p^{\prime}}}\right)\sin(j\alpha\log q).
\end{equation}
Since $1\leq p<2$, from (\ref{gammaj}) it follows that $\gamma_{j}$ maps the line $\{z=\alpha+ i\delta_{p^{\prime}}:-\tau/2<\alpha\leq\tau/2\}$ onto a horizontal ellipse around the origin in the complex plane. In fact, from (\ref{gammaj}) it is clear that $|\gamma_{1}(\alpha+i\delta_{p^{\prime}})|<\gamma_{1}(i\delta_{p^{\prime}})$ for all $\alpha\in(-\tau/2,\tau/2)\setminus\{0\}$ and for these values of $\alpha$, $|\gamma_{j}(\alpha+i\delta_{p^{\prime}})|\leq\gamma_{j}(i\delta_{p^{\prime}})$ for all $j\in\{3,\ldots,2m+1\}$. Implementing these facts in (\ref{phizodd}), we get
\begin{equation}\label{philessodd}
|\phi_{\alpha+i\delta_{p^{\prime}}}(2m+1)|<\phi_{i\delta_{p^{\prime}}}(2m+1),\quad\text{for all }\alpha\in(-\tau/2,\tau/2)\setminus\{0\}\text{ and for all }m\in\mathbb{Z}_{+}.
\end{equation}
Furthermore, from (\ref{gammaj}) it is easy to see that $\gamma_{j}(i\delta_{p^{\prime}})=-\gamma_{j}(\tau/2+i\delta_{p^{\prime}})$ for all $j=1,\ldots,2m+1$. Hence by (\ref{phizodd}), $\phi_{i\delta_{p^{\prime}}}(2m+1)=-\phi_{\tau/2+i\delta_{p^{\prime}}}(2m+1)$. Combining these facts, the assertions (a), (b) and (c) follow for all $n$ odd.

\noindent\textbf{Case 2:} Here we assume that $n$ is even. Once again, we simplify (\ref{generalphiz}) to get
\begin{equation}\label{phizeven}
\phi_{z}(2m)=\begin{cases}
\frac{1}{q+1}\left[\left(1-\frac{1}{q}\right)+\gamma_{2}(z)\right], & \text{ if }m=1,\\
 & \\
\frac{q^{-m}}{q+1}[~(q-1)+q~\gamma_{2m}(z)+(q-1)\sum\limits_{j=1}^{m-1}\gamma_{2j}(z)~], & \text{ if }m\geq 2,
\end{cases}
\end{equation}
where $\gamma_{j}$ is as in (\ref{gammajexplicit}). Observe that unlike Case 1, the first term in $\phi_{z}(2m)$ (see (\ref{phizeven})) is a positive number independent of $z$. Furthermore, using (\ref{gammaj}) we have $|\gamma_{j}(\alpha+i\delta_{p^{\prime}})|\leq\gamma_{j}(i\delta_{p^{\prime}})$ for all $j=4,\ldots,2m$ and $|\gamma_{2}(\alpha+i\delta_{p^{\prime}})|<\gamma_{2}(i\delta_{p^{\prime}})$ for all $\alpha\in(-\tau/2,\tau/2)\setminus\{0,\pm\tau/4\}$. Hence from (\ref{phizeven}), we obtain
\begin{equation}\label{philesseven1}
|\phi_{\alpha+i\delta_{p^{\prime}}}(2m)|<\phi_{i\delta_{p^{\prime}}}(2m),\quad\text{for all }\alpha\in(-\tau/2,\tau/2)\setminus\{0,\pm\tau/4\}\text{ and for all }m\in\mathbb{N}.
\end{equation}
  Using $\gamma_{j}(\pm\tau/4+i\delta_{p^{\prime}})=(-1)^{j/2}\gamma_{j}(i\delta_{p^{\prime}})$ (by (\ref{gammaj})) in (\ref{phizeven}), it is easy to check that $\phi_{\pm\tau/4+i\delta_{p^{\prime}}}(2m)\neq\pm\phi_{i\delta_{p^{\prime}}}(2m)$. Combining this with the fact that $|\phi_{\pm\tau/4+i\delta_{p^{\prime}}}(2m)|\leq\phi_{i\delta_{p^{\prime}}}(2m)$, we get
\begin{equation}\label{philesseven2}
|\phi_{\pm\tau/4+i\delta_{p^{\prime}}}(2m)|<\phi_{i\delta_{p^{\prime}}}(2m),\quad\text{for all }m\in\mathbb{N}.
\end{equation}
Finally, the assertions (a), (b) and (c) follow by using (\ref{phizeven}), (\ref{philesseven1}), (\ref{philesseven2}) and the fact that $\gamma_{j}(i\delta_{p^{\prime}})=\gamma_{j}(\tau/2+i\delta_{p^{\prime}})$ for all $j=2,\ldots,2m$. This completes the proof.

\noindent{\textit{Proof of \rm{(ii)}:}} The proof is similar to that of part (i). Taking modulus on both sides of (\ref{ballsymbol}) and using (\ref{philessodd}), (\ref{philesseven1}) and (\ref{philesseven2}), we obtain
\begin{equation}\label{psiless1}
|\psi_{\alpha+i\delta_{p^{\prime}}}(n)|<\psi_{i\delta_{p^{\prime}}}(n),\quad\text{for all }\alpha\in(-\tau/2,\tau/2)\setminus\{0\}\text{ and for all }n\in\mathbb{N}.
\end{equation}
From (\ref{phizodd}) and (\ref{phizeven}) we also observe that $\phi_{\tau/2+i\delta_{p^{\prime}}}(j)=(-1)^{j}\phi_{i\delta_{p^{\prime}}}(j)$ for all $j\in\mathbb{N}$ and therefore $\psi_{\tau/2+i\delta_{p^{\prime}}}(n)\neq\pm\psi_{i\delta_{p^{\prime}}}(n)$. Implementing this fact in (\ref{ballsymbol}) and using the trivial estimate $|\psi_{\tau/2+i\delta_{p^{\prime}}}(n)|\leq\psi_{i\delta_{p^{\prime}}}(n)$, it follows that
\begin{equation}\label{psiless2}
|\psi_{\tau/2+i\delta_{p^{\prime}}}(n)|<\psi_{i\delta_{p^{\prime}}}(n),\quad\text{for all }n\in\mathbb{N}.
\end{equation}
Finally, combining (\ref{psiless1}) with (\ref{psiless2}), we get the desired result.
\end{proof}

\begin{Remark}
From  (\ref{eqsf}) and (\ref{harishchandra}), we have the following alternative expression for $\phi_z(n), z\in(-\tau/2,\tau/2)\setminus\{0\}, n\in\mathbb{N}$:
\begin{equation}\label{alternativephiexpression}
\phi_{z}(n)=\frac{q^{-n/2}}{q+1}\left[\frac{q\sin((n+1)z\log q)-\sin((n-1)z\log q)}{\sin (z\log q)}\right].
\end{equation}
From (\ref{eqsf}) and the expressions above it is easy to see that the $\tau$-periodic, continuous function $z\mapsto\phi_{z}(n)$ is real-valued for all $z\in S_{2}$. Moreover, if $n\in \N$ is odd then from (\ref{eqsf}) we obtain $\phi_{0}(n)>0$ and $\phi_{\tau/2}(n)<0$. On the other hand, when $n\in\mathbb{N}$ is even, from (\ref{alternativephiexpression}) we have $\phi_{\tau/(4n+4)}(n)>0$ and $\phi_{\tau/(2n+2)}(n)<0$. Therefore by the intermediate value theorem it follows that $\phi_{z}(n)=0$ for some $z\in S_{2}\subseteq S_{p}$. Applying a similar reasoning to the function
$$\psi_{z}(n)=\frac{q^{-n/2}}{q+1}\left[\frac{q^{1/2}\sin(nz\log q)+q\sin((n+1)z\log q)}{\sin (z\log q)}\right],$$
we conclude that the minimum moduli of $z\mapsto\phi_{z}(n)$ and $z\mapsto\psi_{z}(n)$ on $S_{p}$ ($1\leq p<2$) are zero. This  violates the hypothesis of Theorem \ref{minimumcase}. Thus it is not possible to formulate an analogue of Theorem \ref{minimumcase} for spherical and ball averages. This also justifies the use of a bi-infinite sequence in Corollary \ref{pcaseaveragingoperators} instead of a backward infinite sequence.

\end{Remark}

\subsection{The heat operator}
Our next aim is to prove a version of Strichartz's theorem for the complex-time heat operator $\mathscr{H}_{\xi}$, $\xi\in \C^\times$, defined by $\mathscr{H}_{\xi}f=f\ast h_{\xi}$,  where  $\widehat{h}_{\xi}(z)=e^{\xi\gamma(z)}$ (see \cite{AS} for the case $\Re\xi\leq 0$). Clearly, $\mathscr{H}_{\xi}$ includes both the heat semigroup $e^{-t\mathcal{L}}$ (see \cite{CMS1}) and the Schr{\"o}dinger semigroup $e^{it\mathcal{L}}$. We refer to \cite[Theorem 4.0.5]{NS} where in a similar vein for the Riemannian symmetric spaces, only the heat semigroup is considered. We need the following  notation to state our result: For $1\leq p<2$ and  $\xi\in \C^\times$, we define
$$\Phi_{p}(\xi)=(1-\gamma(i\delta_{p^{\prime}}))\cdot((\Re \xi)^{2}+\tanh^{2}(\delta_{p^{\prime}}\log q)(\Im \xi)^{2})^{1/2}.$$
Furthermore, let $\beta_{j}$, $j=1,2$, denote the unique points in $(-\tau/2,\tau/2]$ that satisfy the relation:
\begin{equation}\label{beta}
\Phi_{p}(\xi)\cos\beta_{j}=(-1)^{j} \Re\xi\cdot(1-\gamma(i\delta_{p^{\prime}})),~~\Phi_{p}(\xi)\sin\beta_{j}
=(-1)^{j}\Im\xi\cdot\gamma(\tau/4+i\delta_{p^{\prime}}),\quad\text{for }j=1,2.
\end{equation}

\begin{Corollary}
Fix $\xi\in \C^\times$. For $1\leq p<2$, let $\{f_{k}\}_{k\in\mathbb{Z}}$ be a bi-infinite sequence of functions on $\mathcal{X}$ such that $\|f_{k}\|_{L^{p^{\prime},\infty}(\mathcal{X})}\leq M$ for all $k\in\mathbb{Z}$ and for some $M>0$.
\begin{enumerate}
\item[(i)] If $\mathscr{H}_{\xi}f_{-k}=A~f_{-k+1}$ for all $k\in\mathbb{N}$ and for some $A\in\mathbb{C}$ satisfying $|A|=\exp\{\Re\xi+\Phi_{p}(\xi)\}$, then $\mathcal{L}f_{0}=\gamma(z_{1})f_{0}$, where $z_{1}=\beta_{1}+i\delta_{p^{\prime}}$ and $\beta_{1}$ is as in (\ref{beta}).
\item[(ii)] If $\mathscr{H}_{\xi}f_{k}=A~f_{k+1}$ for all $k\in\mathbb{Z}_{+}$ and for some $A\in\mathbb{C}$ satisfying $|A|=\exp\{\Re\xi-\Phi_{p}(\xi)\}$, then $\mathcal{L}f_{0}=\gamma(z_{2})f_{0}$, where $z_{2}=\beta_{2}+i\delta_{p^{\prime}}$ and $\beta_{2}$ is as in (\ref{beta}).
\end{enumerate}
In both cases, $f_{0}=\mathcal{P}_{z_{j}}F_{j}$ for a unique $F_{j}\in L^{p^{\prime}}(\Omega)$, where $j=1,2$.
\end{Corollary}

\begin{proof}
Since $\widehat{h}_{\xi}(z)\neq 0$ for all $z\in S_{p}$ and for all $1\leq p< 2$, we shall use Theorem \ref{pcaseforfunctionsoflaplacianmaximumnonzero} to get the desired conclusions. Let us define
$$h(\alpha)=\Re\xi\cdot\Re\gamma(\alpha+i\delta_{p^{\prime}})-\Im\xi\cdot\Im\gamma(\alpha+i\delta_{p^{\prime}}),\quad\text{for all }\alpha\in(-\tau/2,\tau/2].$$
A straightforward computation using (\ref{realimaginarygamma}) shows  that the maximum and the minimum values of $h$ on the interval $(-\tau/2,\tau/2]$ are $\Re\xi+\Phi_{p}(\xi)$ and $\Re\xi-\Phi_{p}(\xi)$ respectively. Since $z\mapsto\widehat{h}_{\xi}(z)$ is an even, $\tau$-periodic and entire function, using maximum modulus principle, it follows that
\begin{align*}
\max\limits_{z\in S_{p}}|e^{\xi\gamma(z)}|&=\max\limits_{-\tau/2<\alpha\leq\tau/2}e^{h(\alpha)}=\exp\{\Re\xi+\Phi_{p}(\xi)\},\quad\text{and}\\
\min\limits_{z\in S_{p}}|e^{\xi\gamma(z)}|&=\min\limits_{-\tau/2<\alpha\leq\tau/2}e^{h(\alpha)}=\exp\{\Re\xi-\Phi_{p}(\xi)\}.
\end{align*}
Furthermore, the maximum and minimum is attained at the points $z_{1}$ and $z_{2}$ respectively, for $z_1, z_2$ as in the hypothesis of this corollary. This proves both (i) and (ii).
\end{proof}

\subsection{Functions satisfying  Hardy-type norm-estimates}\label{remark on hardy type functions}
So far in this section we have taken weak $L^p$-function on $\mathcal X$ as an example of $L^p$-tempered distribution and proved all possible variants of Strichartz's theorem for such functions. In fact, while proving the main results of this section, that is, Theorems \ref{pcaseforfunctionsoflaplacianmaximumzero} and \ref{pcaseforfunctionsoflaplacianmaximumnonzero}, we have used Propositions \ref{multiplieronfunctionspaces}, \ref{regulardistributions}, \ref{derivativeestimates}, Remark \ref{remark-on-weak-lp} and Theorem \ref{weaklpchar}, whose proofs rely heavily on the properties of the weak $L^{p}$ spaces. We shall now see that in all these preparatory results one can substitute weak $L^p$-functions by functions which satisfy  Hardy-type norm estimates, which we shall define below. Finally, a meticulous step by step adaptation of the arguments of the proofs of Theorems \ref{pcaseforfunctionsoflaplacianmaximumzero} and \ref{pcaseforfunctionsoflaplacianmaximumnonzero} along with the following results will lead to analogous outcomes for functions that satisfy Hardy-type estimates.

\vspace*{0.1in}

\noindent{\textbf 1.} Fix $1\leq p<2$. Let $\{\omega_0, \omega_1, \ldots\}$ be a fixed infinite geodesic ray starting from the reference point $o$ in $\mathcal{X}$ such that $d(o,\omega_{n})=n$ for all $n\in\mathbb{Z}_{+}$. We define $H^{r}_{p}(\mathcal{X})$ to be the space of all complex-valued functions $f$ on $\mathcal{X}$ for which
\begin{equation}\label{hardy}
\|f\|_{H^{r}_{p}(\mathcal{X})}:=\begin{cases}
~~\sup\limits_{n\in\mathbb{Z}_{+}}~\phi_{i\delta_{p}}(\omega_{n})^{-1}\left(~\int\limits_{K}|f(k\cdot\omega_n)|^{r}dk~\right)^{1/r}<\infty, & \text{ if }1\leq r<\infty,\\
& \\
~~\sup\limits_{x\in\mathcal{X}}~\phi_{i\delta_{p}}(x)^{-1}|f(x)|<\infty, & \text{ if }r=\infty.
\end{cases}
\end{equation}
Since $\phi_{z}(x)=\mathcal{P}_{z}{\bf 1}(x)$, $\phi_{i\Im z}(x)>0$ for all $x\in\mathcal{X}$ and for all $z\in\mathbb{C}$. Therefore the expression on the right hand side of (\ref{hardy}) is well-defined. Furthermore, from the explicit expression (\ref{eqsf}) of $\phi_{z}$, one can easily derive that for to every $1\leq p<2$, there exist positive constants $A_{p}$ and $B_{p}$ such that
\begin{equation}\label{phipointwise1}
A_{p}~q^{-\frac{|x|}{p^{\prime}}}\leq |\phi_{i\delta_{p}}(x)|\leq B_{p}~q^{-\frac{|x|}{p^{\prime}}},\quad\text{for all }x\in\mathcal{X}.
\end{equation}
Now combining (\ref{phipointwise1}) with (\ref{hardy}) we get, for all $f\in H^{r}_{p}(\mathcal{X})$,
\begin{equation}\label{hardy2}
\|f\|_{H^{r}_{p}(\mathcal{X})}\asymp\begin{cases}
~~\sup\limits_{n\in\mathbb{Z}_{+}}~q^{\frac{n}{p^{\prime}}}\left(~\int\limits_{K}|f(k\cdot\omega_n)|^{r}dk~\right)^{1/r}, & \text{ if }1\leq r<\infty,\\
& \\
~~\sup\limits_{x\in\mathcal{X}}~q^{\frac{|x|}{p^{\prime}}}|f(x)|, & \text{ if }r=\infty.
\end{cases}
\end{equation}

\noindent{\textbf 2.} We also have the following characterization of eigenfunctions of $\mathcal{L}$ that satisfy the Hardy-type estimates (see \cite{R4} for details).

\begin{Theorem}[{{\cite[Theorem 1.2]{R4}}}]\label{hardypcase}
Let $1\leq p<2$ and $f$ be a complex-valued function on $\mathcal{X}$. Suppose that $z=\alpha+i\delta_{p^{\prime}}$, where $\alpha\in\mathbb{R}$. Then $f=\mathcal{P}_{z}F$ for some $F\in L^{r}(\Omega)$ if and only if $f\in H^{r}_{p}(\mathcal{X})$ for $1<r\leq\infty$ and $\mathcal{L}f=\gamma(z)f$. When $r=1$, the above result holds true with $f=\mathcal{P}_{z}\mu$ for some complex measure $\mu$ on the boundary $\Omega$.
\end{Theorem}

\noindent{\textbf 3.} Our next aim is to prove that an analogue of Proposition \ref{multiplieronfunctionspaces} holds on $H^{r}_{p}(\mathcal{X})$ for all $1\leq p<2$ and $1\leq r\leq\infty$. To see this we consider $\phi\in\mathcal{S}_p(\mathcal{X})^{\#}$. For a finitely supported function $f$ on $\mathcal{X}$,  we have (see (\ref{convolution}))
$$\left(~\int\limits_{K}|f\ast\phi(k\cdot\omega_{n})|^{r}dk\right)^{1/r}=\left(~\int\limits_{K}\left|~\sum\limits_{y\in\mathcal{X}}f(k\cdot y)~\phi(d(\omega_{n},y))~\right|^{r}dk~\right)^{1/r},\quad\text{for all }1\leq r<\infty.$$
Applying the Minkowski's integral inequality to the right hand side of the  expression above and further using (\ref{phipointwise1}), (\ref{hardy2}), it follows that
$$\phi_{i\delta_{p}}(\omega_{n})^{-1}\left(~\int\limits_{K}|f\ast\phi(k\cdot\omega_{n})|^{r}dk~\right)^{1/r}\leq\nu_{p,m}(\phi)\|f\|_{H^{r}_{p}(\mathcal{X})}\mathlarger{\mathlarger{\mathlarger\sum}}\limits_{y\in\mathcal{X}}~\frac{q^{(d(o,\omega_{n})/p^{\prime}-d(o,y)/p^{\prime}-d(\omega_{n},y)/p)}}{(1+d(\omega_{n},y))^{m}}.$$
Plugging in formula (\ref{confluenceformula}) with $x=\omega_{n}$, we obtain
$$\phi_{i\delta_{p}}(\omega_{n})^{-1}\left(~\int\limits_{K}|f\ast\phi(k\cdot\omega_{n})|^{r}dk~\right)^{1/r}\leq\nu_{p,m}(\phi)\|f\|_{H^{r}_{p}(\mathcal{X})}\mathlarger{\mathlarger{\mathlarger\sum}}\limits_{y\in\mathcal{X}}~\frac{q^{-2d(c(\omega_{n},y),y)/p^{\prime}}q^{(1-2/p)d(\omega_{n},y)}}{(1+d(\omega_{n},y))^{m}}.$$
It suffices to prove that the infinite sum in the right hand side of the expression above is bounded by some positive constant $C$, independent of $n$. To get this estimate, we apply a similar reasoning as in the proof of Proposition \ref{lp_multiplier} and decompose $\mathcal{X}$ as the disjoint union of the sets $G_{j,l}(\omega_{n})$ (defined by (\ref{gjx}) and (\ref{gjnx})) so that
$$\mathlarger{\mathlarger{\mathlarger\sum}}\limits_{y\in\mathcal{X}}~\frac{q^{-2d(c(\omega_{n},y),y)/p^{\prime}}q^{(1-2/p)d(\omega_{n},y)}}{(1+d(\omega_{n},y))^{m}}\leq \mathlarger{\mathlarger{\mathlarger\sum}}\limits_{j=0}^{\infty}q^{(1-2/p)j}\mathlarger{\mathlarger{\mathlarger\sum}}\limits_{l=0}^{\infty}\frac{1}{(1+l)^{m}}\leq C,\quad\text{whenever }m\geq 2.$$
This completes the proof when $1\leq r<\infty$. By slightly modifying these arguments one can also prove the case $r=\infty$. Therefore $H^{r}_{p}(\mathcal{X})\ast\mathcal{S}_p(\mathcal{X})^{\#}\subseteq H^{r}_{p}(\mathcal{X})$ and for every $\phi\in\mathcal{S}_p(\mathcal{X})^{\#}$, there exists a seminorm $\nu_{p,m}(\cdot)$ of $\mathcal{S}_p(\mathcal{X})$ and a constant $C>0$ such that
\begin{equation}\label{multiplierhardytype}
\|f\ast\phi\|_{H^{r}_{p}(\mathcal{X})}\leq C~\nu_{p,m}(\phi) \|f\|_{H^{r}_{p}(\mathcal{X})},\quad\text{for all }f\in H^{r}_{p}(\mathcal{X}).
\end{equation}

\noindent{\textbf 4.} In Proposition \ref{generalizedproposition} if we assume that (\ref{generalized2}) holds for some $f\in H^{r}_{p}(\mathcal{X})$ instead of a $L^{p}$-tempered distribution $T$ on $\mathcal{X}$ then a step by step adaptation of the arguments in Remark \ref{remark-on-weak-lp} (1) together with the estimate (\ref{multiplierhardytype}) (instead of Proposition \ref{multiplieronfunctionspaces}) will lead us to similar conclusions with the only difference that $f_{i}\in H^{r}_{p}(\mathcal{X})$ for all $i=1,\ldots,j$. If we further assume that $f\in H^{r}_{p}(\mathcal{X})$ is radial then for each $i$, $f_{i}$ is also radial.

\noindent{\textbf 5.} Let $1\leq p_{1}\leq p_{2}<2$ and $1\leq r\leq \infty$. We will now prove that for $f\in H^{r}_{p_{2}}(\mathcal{X})$, there exists a semi-norm $\nu_{p_{1},m}(\cdot)$ of $\mathcal{S}_{p_{1}}(\mathcal{X})$ and a constant $C>0$ such that
$$|\langle f,\phi \rangle|\leq C\|f\|_{H^{r}_{p_{2}}(\mathcal{X})}~\nu_{p_{1},m}(\phi),\quad\text{for all }\phi\in\mathcal{S}_{p_{1}}(\mathcal{X}).$$
To see this we assume that $\phi\in\mathcal{S}_{p_{1}}(\mathcal{X})^{\#}$. Now fix $m\geq 10$ and observe that $\nu_{p_{1},m}(\phi)<\infty$ (see (\ref{schwartzspace})). Since $\# S(o,n)\asymp q^{n}$, it follows that
$$|\langle f,\phi \rangle|\leq\nu_{p_{1},m}(\phi)\sum\limits_{x\in\mathcal{X}}\frac{q^{-|x|/p_{1}}}{(1+|x|)^{m}}|f(x)|\leq\nu_{p_{1},m}(\phi)\sum\limits_{n=0}^{\infty}\frac{q^{n/p^{\prime}_{1}}}{(1+n)^{m}}\left(~\int\limits_{K}|f(k\cdot\omega_{n})|^{r}~dk~\right)^{1/r}.$$
Since $f\in H^{r}_{p_{2}}(\mathcal{X})$ with $p_{1}\leq p_{2}$, using (\ref{hardy2}) we get
$$|\langle f,\phi \rangle|\leq\nu_{p_{1},m}(\phi)~\|f\|_{H^{r}_{p_{2}}(\mathcal{X})}\sum\limits_{n=0}^{\infty}(1+n)^{-m}q^{n\left(\frac{1}{p^{\prime}_{1}}-\frac{1}{p^{\prime}_{2}}\right)}\leq C\|f\|_{H^{r}_{p_{2}}(\mathcal{X})}~\nu_{p_{1},m}(\phi),$$
which proves our assertion. Thus elements of $H^{r}_{p_{2}}(\mathcal{X})$ are $L^{p_1}$-tempered distributions on $\mathcal{X}$ for $1\leq p_{1}\leq p_{2}<2$ and an analogue of Proposition \ref{regulardistributions} holds on $H^{r}_{p}(\mathcal{X})$.

\noindent{\textbf 6.} Finally, arguing as in the proof of Proposition \ref{derivativeestimates} and then substituting the pointwise estimate (\ref{derivativepointwise}) into (\ref{hardy2}), it is easy to see that if for a polynomial $P$ with complex coefficients and $z_{0}=\alpha+i\delta_{p^{\prime}}$, $P(D)\phi_{z}|_{z=z_{0}}\in H^{r}_{p}(\mathcal{X})^{\#}$ for some $1\leq p<2$ and $1\leq r\leq\infty$, then $P$ is a constant polynomial.


\section*{Acknowledgement}

Sumit Kumar Rano gratefully acknowledges the support provided by the National Board of Higher Mathematics (NBHM) post-doctoral fellowship (Number: 0204/3/2021/R$\&$D-II/7386) from the Department of Atomic Energy (DAE), Government of India.


%
%
%


\end{document}